\documentclass[12pt]{amsart}

\usepackage{hyperref}
\usepackage{amsmath,amssymb,latexsym,amsthm}
\usepackage{url}
\usepackage{graphicx,color}

\begin{document}

\def \cal {\mathcal}

\newcommand{\refeq}[1]{(\ref{#1})}
\def \beq {\begin {eqnarray}}
\def \eeq {\end {eqnarray}}
\def \ba {\begin {eqnarray*}}
\def \ea {\end {eqnarray*}}

\newtheorem{theorem}{Theorem}[section]
\newtheorem{lemma}[theorem]{Lemma}
\newtheorem{proposition}[theorem]{Proposition}
\newtheorem {definition}[theorem]{Definition}
\newtheorem {example}[theorem]{Example}
\newtheorem {corollary}[theorem]{Corollary}
\newtheorem {remark}[theorem]{Remark}

\def\R{{\mathbb R}}
\def\Z{{\mathbb Z}}
\def\S{{\mathbb S}}
\def\B{{\mathbb B}}
\def\C{{\mathbb C}}
\def\expec{{\mathbb E}}

\def\O{{\mathcal O}}

\def\re{{\rm Re}\,}
\def\im{{\rm Im}\,}

\def\bra{\langle}
\def\cet{\rangle}
\def\p{\partial}

\def\hat{\widehat}
\def\tilde{\widetilde}

\newcommand{\note}[1]{\marginpar{\footnotesize #1}}
\newcommand{\anote}[1]{\marginpar{{\color{red}\footnotesize #1}}}
\def\I{\mathcal{I}}
\def\rf{\lambda}
\def\df{\rho}
\def\dom{D}
\def\plane{\R^3_0}
\def\hs{\R^3_+}
\def\md{U}
\def\s{\sigma}
\newcommand{\Ft}{{\mathcal F}}

\def\cspace{\Theta}

\newcommand{\vect}[2]{\begin{pmatrix}
#1 \\ #2
\end{pmatrix}}

\newcommand{\Rt}{{\mathcal S}}
\newcommand{\btheta}{\boldsymbol\theta}

\newcommand{\norm}[1]{\left\|#1\right\|}

\newcommand{\pw}{\widetilde{W}}

\newcommand{\sset}{X}

\title[Inverse acoustic scattering problem in half-space]{Inverse acoustic scattering problem in half-space with anisotropic random impedance}
\author {Tapio Helin, Matti Lassas and Lassi P\"{a}iv\"{a}rinta}
\date{\today}
\maketitle

\begin{abstract}
We study an inverse acoustic scattering problem in half-space with a probabilistic impedance boundary value condition. The Robin coefficient (surface impedance) is assumed to be a Gaussian random function with 
a pseudodifferential operator describing the covariance.
We measure the amplitude of the backscattered field averaged over the frequency band and assume that the data is generated by a single realization of $\rf$. Our main result is to show that under certain conditions the principal symbol of the covariance operator of $\rf$ is uniquely determined. Most importantly, no approximations are needed and we can solve the full non-linear inverse problem.
We concentrate on anisotropic models for the principal symbol, which leads to the analysis of a novel anisotropic spherical Radon transform and its invertibility.
\end{abstract}

\section{Introduction}

In this work we study inverse acoustic scattering in half-space. 
We assume that the time-harmonic acoustic field $u$ satisfies the Helmholtz equation
\begin{equation}
	\label{eq:helmholtz}
	(\Delta+k^2) u(x) = \delta_y(x), \quad x = (x_1,x_2,x_3) \in \hs,
\end{equation}
where $\hs = {\R^2 \times (0,\infty)}$, $k\in\R_+$ is the wave number and $\delta_y$ is the Dirac delta distribution at $y\in\hs$, i.e., the propagating wave is generated by a point source located in the upper half-space.
Moreover, the total field $u = u(\cdot; y,k)$ is assumed to satisfy the impedance boundary value condition
\begin{equation}
	\label{eq:robin_boundary}
	\frac{\p u}{\p x_3}(x) + \rf_k u(x) = 0
\end{equation}
on $\plane := {\R^2 \times \{0\}}$, where $\rf_k=\rf_k(x)$ is an unknown realization of a real-valued random function with a bounded support. We assume that the wave number $k$ is positive and $\rf_k$ is real-valued. Notice that in our model $\rf_k$ depends on $k$.

The classical problem with impedance boundary value condition in the half-space geometry is to understand what kind of
surface waves appear on $\plane$. Related to this, 
the uniqueness of the solution in many cases requires a special
radiation condition \cite{Nedelec_3d, PZ13}. In our case it can be shown that the 
classical Sommerfeld radiation condition 
\begin{equation}
	\label{eq:sommerfeld_radiation}
	\left(\frac{\p}{\p r} -ik\right) u(x) = o(|x|^{-1}), \quad {\rm as} \;\; |x|\to\infty
\end{equation}
and uniformly in the sphere $x/|x| \in \S^2$, guarantees the uniqueness for a real-valued and compactly supported $\rf_k$. 

In the context of acoustics and sound propagation, the parameter $\rf_k$ is typically factorized as $\rf_k = ik\beta$,
where $\beta$ is the acoustic admittance of the surface.
It describes the ratio between the normal fluid velocity and the pressure at the surface. In this work we require $\re(\beta) = 0$ in order to fulfil assumptions on $\rf_k$. In acoustics, the boundary is said to be passive or non-absorbing.

Let $\md$ be an open and bounded set in $\hs$. Our interest lies in the following inverse problem: 
\emph{given the back-scattered field $u(y; y,k)$ for all $y\in\md$ and $k>0$, what information can be recovered regarding $\rf_k$?}
In other words, we take measurements generated by a single realization of $\rf_k$ and
ask what properties the underlying random process had.
The role of randomness in this treatise is to model the complex or chaotic micro-structure of $\rf_k$. We do not focus on recovering $\rf_k$ exactly, but instead work towards determining some statistical properties regarding its probability distribution.
We return to a more detailed formulation of this problem below.

Our work draws inspiration from \cite{LPS}, where inverse scattering was studied for a two-dimensional random Schr\"{o}dinger equation $(\Delta+q+k^2)u = 0$. The potential $q$ was assumed to be a Gaussian random function such that the covariance operator is a classical pseudodifferential operator \cite{Hor3}. The result in \cite{LPS} shows that
the backscattered field, obtained from a single realization of $q$, determines uniquely the principal symbol of the covariance operator of $q$.
The statistical model for $q$ assumes that the potential is locally isotropic and that the smoothness remains unchanged in spatial changes. However, the local variance is allowed to vary. A random field with such properties was called \emph{microlocally isotropic}. This large class of random fields includes stochastic processes like
the Brownian bridge or the Levy Brownian motion in the plane.

In the present treatise we generalize this concept to a class of random fields that are called \emph{microlocally anisotropic}. Similar to \cite{LPS} the covariance operator
is assumed to be a pseudo-differential operator. 
However, the principal symbol is allowed to be direction-dependent. Hence, the correlation of the field is anisotropic while the smoothness is remains unchanged in spatial changes. 

Our main results in Theorems \ref{thm:main_res_iso} and \ref{thm:main_result} relate to the unique recovery of the principal symbol \cite{Hor3} of the covariance of $\rf$. In the isotropic case, the principal symbol can be fully recovered. If the field is anisotropic, a partial recovery is always possible. In particular, if the degrees of freedom can be reduced, e.g., the field is solenoidal, the anisotropic principal symbol can be uniquely determined. Most importantly, no approximations are made and we can study the full non-linear inverse problem. Further, the model of the anisotropic random field leads to the analysis of a novel anisotropic spherical Radon transform and its invertibility.

The forward problem related to \eqref{eq:helmholtz}-\eqref{eq:robin_boundary} has been widely studied in relation to outdoor sound propagation: how to predict the far field behaviour of the sound field emitted from a monofrequency point source located above? The problems related to energy-absorbing boundaries with $\re \beta > 0$ have been studied in detail by Chandler-Wilde (see e.g. \cite{CW97,CW_Peplow05}). The energy propagating at boundary level is maximized when $\re \beta = 0$ and $\im \beta < 0$. In this case, the outgoing surface wave decays slower than volume waves and, consequently, special radiation conditions need to be considered. The speed of surface waves were studied in \cite{Nedelec_3d,Nedelec_3d_2,Nedelec_2d,Karamyan02,Karamyan03,Karamyan04} in detail. We also mention \cite{Odeh63} on uniqueness results when the surface waves do not exist.

The direct problem related to random Robin boundary conditions in half-space for the Helmholtz equation was considered in \cite{Bal11} in the context of homogenization theory. The aim is to find an effective solution to the problem when the oscillations increase.
Often, however, randomness in inverse problems is related to the source \cite{HeLaOk2,Bao14} or the medium \cite{Solna_book}.
Typical approach to scattering from a random medium relies on the multiscale asymptotics of the scattered field. This direction has been studied by Papanicolaou and others in various cases \cite{Bal,Tsogka}. 
Moreover, scattering effects from random boundary were studied in \cite{Borcea13a}.

This paper is organized as follows. In section \ref{sec:statement} we describe the main results in detail. The probabilistic model for the random Robin parameter $\rf_k$
is introduced and motivated in section \ref{sec:random_field}. Despite the forward scattering problem being classical, the problem \eqref{eq:helmholtz}-\eqref{eq:robin_boundary} in half-space geometry is scarcely considered in the literature and we include a rigorous analysis of the situation in section \ref{sec:forward}.
Assuming the Born approximation, we analyse the cross-correlations in the back-scattered data. In section \ref{sec:born} we show an asymptotic formula for
the cross-correlations, while in section 6 ergodicity arguments are used to prove the convergence of the measurement. Finally, in section \ref{sec:recoverability} we describe the recoverability of the anisotropic microlocal strength.

\section{Statement of the result}
\label{sec:statement}

Throughout the paper we identify $\plane = \R^2\times \{0\}$ with $\R^2$. In consequence, whenever
an object is defined on $\R^2$ this should be interpreted as the boundary of the half-space. However, when elements from $\hs$ and the boundary appear simultaneously in a equation we prefer to distinguish the boundary by $\plane$.
First example of this convention is the random Robin parameter $\rf$ supported in $\dom \subset \R^2$. 


We assume the probability space 
$(\Omega, {\mathcal F}, \mathbb{P})$ is complete and that the Robin parameter $\rf$ is a zero-centered generalized Gaussian random field on $\plane$ with a covariance operator $C_\rf : C^\infty_0(\R^2) \to {\mathcal D}'(\R^2)$. 
Furthermore, we assume that $C_\rf$ is a classical pseudodifferential operator.
We return to a rigorous definition of generalized random fields in Section \ref{sec:random_field}.

A few words on notation: throughout the text we denote by $\S^{n-1}$ and $\B^n$ the unit sphere and ball in $\R^n$, respectively. In addition, we write e.g. $\S^{n-1}(R)$ to distinguish a sphere with radius $R$. Moreover, we use convention $z^0 = z/|z| \in \S^1$ for any $z\in \plane$. For two functions $f=f(x)$ and $g=g(x)$ we write $f \propto g$ if there exists a constant $C$ such that $f = Cg$ everywhere. Further, $C$ denotes a generic constant the value of which can change even inside a formula.

We distinguish microlocally isotropic and anisotropic fields by the following definitions.

\begin{definition}
\label{def:iso_rf}
The random field $\rf$ is called \emph{microlocally isotropic} (of order $\epsilon$) in $\dom$, if the principal symbol of $C_\rf$ satisfies
$\s^p(x,\xi) = b(x)|\xi|^{-2\epsilon-2}$,
where $b \in C_0^\infty(\R^2)$ is supported in $\dom$.
\end{definition}
\begin{definition}
\label{def:aniso_rf}
The random field $\rf$ is called \emph{microlocally anisotropic} (of order $\epsilon$) in $\dom$, if the principal symbol of $C_\rf$ satisfies
\begin{equation}
	\label{eq:aniso_prin}
	\s^p(x,\xi) = \frac{b(x,\xi^0)}{|\xi|^{2\epsilon+2}}, \quad  \xi^0 = \frac{\xi}{|\xi|},
\end{equation}
for $\epsilon > 0$ and some bounded function $b \in C^\infty(\R^2\times \S^1)$ supported in $\dom \times \S^1$. 
\end{definition}
The function $b$ appearing in Definitions \ref{def:iso_rf} and \ref{def:aniso_rf} is called \emph{isotropic} and \emph{anisotropic local strength} of $\rf$, respectively. In the following sections we consider the isotropic model as
a special case of anisotropy and hence the notation $b$ always depends on $\xi$.

Let us now record main technical assumptions regarding our model.
In order to establish the main results, the anisotropic local strength $b$ of $\lambda$ is assumed to satisfy following conditions:
\begin{itemize}
	\item[(A1)] $b(x,\xi^0) = b(x,-\xi^0)$ for any $\xi^0 \in \S^1$ and
	\item[(A2)] there exists $\tilde b \in C^\infty(\R^2 \times \R^2)$ and $s>0$ such that $\tilde b(x,\cdot)$ is real-analytic and $s$-homogeneous, i.e., 
	$\tilde b(x,y) = |y|^s b(x,y^0)$ for any $x,y\in \R^2$.
\end{itemize}
The assumption (A2) allows us to use unique continuation of analytic functions in the proof of Theorem \ref{thm:recover}. 

Recall now that the set of measurement points $\md$ is an open and bounded set in $\hs$. We assume that its projection to $\plane$ is disjoint to $\dom$, i.e., 
\begin{itemize}
	\item[(A3)] $\md'=\{(x',0)\in \plane \; | \; (x',x_3)\in \md\}$ satisfies $\md' \cap \dom = \emptyset$.
\end{itemize}
Moreover, the random field $\rf$ is assumed to depend on $k$ according to
\begin{itemize}
	\item[(A4)] $\rf_k(x) = \frac{\rf(x)}{k^p}$ for $k\geq 1$, $p>\epsilon+\frac 12$
	and $\rf$ is a microlocally anistropic random field.
\end{itemize}
The assumption (A3) is technical in nature and is required in section \ref{sec:born} (see remark \ref{re:rho_wellposed}) and in the proof of Theorem \ref{thm:recover}. The assumption (A4) is related to the convergence speed of the Born series and is discussed in remark \ref{re:freq_dependency} in more detail.

In the scattering problem the wave $u$ is decomposed as
\begin{equation*}
	u(x; y,k) = u_{in}(x; y,k) + u_s(x; y,k),
\end{equation*}
where $u_s$ is the scattered field and
\begin{equation}
	u_{in}(x;y,k) = g_k(x-y) + g_k(\tilde x-y).
\end{equation}
Above, we have $\tilde x = (x_1, x_2, -x_3)$ and $g_k$ stands for
$$g_k(x) = \frac {\exp(ik|x|)}{4\pi|x|}.$$ 
We point out that 
$\left.\frac{\p}{\p x_3} g_k(x)\right|_{x_3=0}=0$ for any $x \in \plane$ and, consequently, $\left.\frac{\p}{\p x_3} u_{in}(\cdot;y,k)\right|_{x_3=0} = 0$ for any $y\in \hs$. Thus, $u_{in}$ is the solution to the Helmholtz problem $\eqref{eq:helmholtz}$ and $\eqref{eq:sommerfeld_radiation}$ in half-space geometry with zero Neumann boundary condition.

We are ready describe our measurement data.

\begin{definition}
\label{def:measurement}
Given $\omega \in \Omega$ and $x,y \in \md$, the measurement $m(x,y,\omega)$ is the pointwise limit
\begin{equation}
	\label{eq:measurement}
	m(x,y,\omega) = \lim_{K\to \infty} \frac{1}{K-1} \int_1^K k^{2(1+\epsilon+p)}|u_s(x; y,k,\omega)|^2 dk.
\end{equation}
We call $m(y,y,\omega)$ the backscattering measurement.
\end{definition}
The well-posedness of the direct problem is shown by the following result.
\begin{theorem}
\label{thm:data}
Let $\dom \subset \plane$ be a bounded simply connected domain and $U\subset \hs$ the measurement set satisfying assumption (A3). Moreover, let $\rf$ be a microlocally anisotropic Gaussian random field of order $\epsilon>0$ in $\dom$ such that assumptions (A1), (A2) and (A4) are satisfied. Then the following holds:
\begin{itemize}
\item[(i)] For any $x,y \in \md$ the measurement $m(x,y,\omega)$ is well-defined (the limit in \eqref{eq:measurement} exists almost surely), if
the line segment $L_{x,y} = \{tx+(1-t)y\; | \; t\in [0,1]\}$ lies in the exterior of $\dom$, i.e., $L_{x,y} \subset \plane \setminus \overline{\dom}$.
\item[(ii)] There exists a continuous deterministic function $m_0(x,y)$ such that for any
$x,y\in \md$ the equality $m(x,y,\omega) = m_0(x,y)$ holds almost surely. In particular, the function $n_0(x) := m_0(x,x)$ is almost surely determined by the backscattering data $\{m(x,x,\omega) \; : \; x \in \md\}$.
\item[(iii)] The backscattering data i.e. $n_0(x)$, $x = (x',x_3) \in \md$ is determined by the micro-correlation strength $b$ through the relation
\begin{equation}
	n_0(x) = \frac{1}{2^{8+2\epsilon}\pi^2} \int_{\dom} \frac{b(z,(z-x')^0)}{|z-x|^2} dz.
\end{equation} 
\end{itemize}
\end{theorem}
As an intermediate step, we prove in Theorem \ref{thm:recover} that
the backscattering data $n_0(x)$, $x = (x',x_3) \in \md$ determines the integrals
\begin{equation*}
	(\Rt b)(x',r) = \int_{\S^1} b(x'+r\theta,\theta) d|\theta|
\end{equation*}
for all $x'\in\plane$ and $r>0$.
We call $\Rt$ an anisotropic spherical Radon transform. To the knowledge of the authors, the invertibility of $\Rt$ has not been studied in literature.

Notice that if $\rf$ is isotropic, $\Rt$ reduces to the standard spherical Radon transform. In this case, the question of invertibility is classical (see e.g., \cite{AQ96} and the references therein) and follows directly. Hence we can formulate our main result on the isotropic model.
\begin{theorem}
\label{thm:main_res_iso}
Let $\dom\subset \plane$ and $U\subset \hs$ be as in Theorem \ref{thm:data}.
Moreover, let $\rf$ be microlocally isotropic of order $\epsilon>0$ in $\dom$.
Then the backscattering data $n_0(x)$, $x\in \md$, uniquely determines $b=b(x)$ everywhere.
\end{theorem}

Regarding the anisotropic problem, we explicitly describe the null-space
of $\Rt$ in Theorem \ref{thm:nullspace}. This yields us the following result on the invertibility of $S$.

\begin{theorem}
\label{thm:invert_S}
Under the assumptions in Theorem \ref{thm:data},
the backscattering data $n_0(x)$, $x\in \md$ uniquely determines values
\begin{equation*}
	(\Ft b)(\xi, \xi^0) \quad {\rm and} \quad (\Ft b)(\xi,(\xi^0)^\bot),
\end{equation*}
for all $\xi \in \plane$, where $\Ft = \Ft_{x\to\xi}$ is the Fourier transform and $\xi^\bot = (\xi_2,-\xi_1).$
\end{theorem}
Our interest lies in the case when the anisotropic local strength is of quadratic form
\begin{equation}
	\label{eq:aniso_is_quad}
	b(x, \xi^0) = \langle \xi^0, A(x)\xi^0\cet,
\end{equation}
where the matrix field $A : \R^2 \to \R^{2\times 2}$ is 
\begin{itemize}
	\item[(A5)]smooth and symmetric, has uniformly bounded eigenvalues and satisfies ${\rm supp}(A) \subset D$. 
\end{itemize}
The main result regarding the quadratic model is then as follows.
\begin{theorem}
\label{thm:main_result}
Let the assumptions in Theorem \ref{thm:data} hold. In addition, 
we assume that the local strength of $\lambda$ is of the form \eqref{eq:aniso_is_quad}, where $A : \R^2 \to \R^{2\times 2}$ satisfies (A5). 
Given the backscattering data $n_0(x)$, $x\in \md$, the trace ${\rm tr}(A)$ can be uniquely determined everywhere.
Moreover, suppose that one of three coefficient functions $a_{j}$, $1\leq j\leq 3$ from 
\begin{equation*}
	A(x) =
	\begin{pmatrix}
		a_1(x) & a_3(x) \\
		a_3(x) & a_2(x)
	\end{pmatrix}
\end{equation*}
is known,
then the backscattering data $n_0(x)$, $x\in\md$, uniquely determines the other two everywhere. 
\end{theorem}

\section{Properties of the random field}
\label{sec:random_field}

\subsection{Smoothness of the realizations}

We begin by defining a generalized Gaussian field. Let $\rf$ be a measurable map from the probability space $\Omega$ to the space of real-valued ${\mathcal D}'(\R^2)$ such that for all $\phi_1, ..., \phi_m \in C^\infty_0(\R^2)$ the
mapping $\Omega \ni \omega \mapsto (\bra \rf(\omega), \phi_j\cet)_{j=1}^m$ is
a Gaussian random variable. The distribution of $\rf$ is
determined by the expectation $\expec \rf$ and the covariance operator
$C_\rf : C^\infty_0(\R^2) \to {\mathcal D}'(\R^2)$ defined by
\begin{equation}
	\label{eq:defcov}
	\bra \psi_1, C_\rf \psi_2 \cet = 
	\expec (\bra \rf - \expec \rf, \psi_1 \cet \bra \rf - \expec \rf, \psi_2 \cet).
\end{equation}
Let $c_\rf(x,y)$ be the Schwartz kernel of the covariance operator $C_\rf$.
We call $c_\rf(x,y)$ the covariance function of $\rf$. Then, in the sense of generalized
functions, \eqref{eq:defcov} reads as
\begin{equation}
	c_\rf(x,y) = \expec ((\rf(x) - \expec \rf(x))(\rf(y) - \expec \rf(y)).
\end{equation}

\begin{proposition}
Let $\epsilon >0$ and random field $\rf$ as in Definition \ref{def:aniso_rf}. We have $\rf \in C^\alpha(\R^2)$ almost surely for all $\alpha \in (0,\epsilon)$.
\end{proposition}

\begin{proof}
Consider a generalized Gaussian random field $Y$ on $\R^2$ with expectation $\expec Y = 0$ and a covariance operator $C_Y = (I-\Delta)^{-1-\epsilon}$. Clearly, the symbol of $C_Y$ satisfies $\sigma(C_Y) = (1+|\xi|^2)^{-1-\epsilon}$, and we have that the random field $\tilde Y = \phi Y$ for any $\phi \in C^\infty_0(\R^2)$ is microlocally isotropic of order $\epsilon$ in the sense of definition \ref{def:iso_rf}. Due to \cite[Thm. 2]{LPS} we have $Y \in H^{\delta,p}_{loc}(\R^2)$ almost surely for all $\delta < \epsilon$ and $1<p<\infty$.

Let us then define a new generalized random field by $\tilde \lambda = \lambda + Y$. It follows that $C_{\tilde \rf}$ is an uniformly elliptic pseudodifferential operator and we can define a square root $C^{1/2}_{\tilde \rf}$ such that $\sigma(C^{1/2}_{\tilde \rf}) \in S^{-1-\epsilon}(\R^2\times \R^2)$ \cite{Taylor}.
Moreover, it is straightforward to see that $\tilde \rf = C^{1/2}_{\tilde \rf} C_Y^{-1/2}Y$ in (probability) distribution and, consequently, $\tilde \rf \in H^{\delta,p}_{loc}(\R^2)$ almost surely for any $\delta< \epsilon$ due to \cite[Prop. 13.6.5]{Taylor3}.
We conclude that $\rf = \tilde \rf - Y \in H^{\delta,p}(\R^2)$ almost surely and the result follows by the Sobolev embedding theorem.
\end{proof}

\subsection{Examples of microlocal anisotropy}

Microlocally isotropic fields were illustrated by examples of fractional Brownian and Markov fields in \cite{LPS}. Inspired by these we present two interesting cases that are motivated by other inverse problem research. As the second example, we present a problem which does not fully satisfy our assumptions (more precisely (A4)) and can be only partially answered by our analysis. However, the case is highly interesting since the anisotropy is generated by an unknown diffeomorphism. This is often the case in geometrical inverse problems.

\subsubsection{Gaussian potential field}

Define a Gaussian random process $Y$ on $\R^2$ by $\expec Y = 0$ and
\begin{equation*}
	c_Y(z_1,z_2) = |z_1 - z_2|^{2+\epsilon} + r(z_1,z_2)
\end{equation*}
for $\epsilon>0$, where $r \in C^\infty(\R^2\times \R^2)$. We have then $\sigma(C_Y) \in S^{-4-\epsilon}_{1,0}(\R^2\times \R^2)$ and for the principal symbol $\sigma_p(C_Y) = |\xi|^{-4-\epsilon}$, $\xi>1$. Let us now define
\begin{equation*}
	q = D_v Y := \left(v(x) \cdot \nabla\right) Y
\end{equation*}
for a vector field $v \in C^\infty_0(\R^2, \R^2)$ such that ${\rm supp}(v) \subset D$.
The random field $q$ inherits a covariance $C_q = D_v^* C_Y D_v : {\mathcal D}(\R^2) \to {\mathcal D}'(\R^2)$. Since $\sigma(D_v) = v(x) \cdot \xi \in S^1_{1,0}(\R^2\times \R^2)$, 
the composition of pseudodifferential operators \cite[Thm. 18.1.8]{Hor3} yields that $\sigma(C_q) \in S^{-2-\epsilon}_{1,0}(\R^2\times \R^2)$ and the principal symbol of $q$ satisfies 
$$\sigma^p(C_q) \propto (v(x)\cdot \xi)^2 |\xi|^{-4-\epsilon} = \bra \xi^0, A(x)\xi^0\cet|\xi|^{-2-\epsilon},$$ where $A(x) = v(x)v(x)^\top$.
By Theorem \ref{thm:main_result} we can recover $v$ if either component is known a priori. Even if this is not known, one can recover ${\rm tr}(A(x)) = \norm{v(x)}^2$. 


\subsubsection{Fractional Brownian fields under diffeomorphism}

Following the Example 1 in \cite{LPS} let us define the multidimensional fractional Brownian motion in $\R^2$ for the Hurst index $H$ as the centered Gaussian process $X_H(z)$ indexed by $z \in \R^2$ with following properties:
\begin{align*}
	& \expec | X_H(z_1) - X_H(z_2)|^2 = |z_1 - z_2|^{2H} \quad \textrm{for all } z_1, z_2 \in \R^2 \\
	& X(z_0) = 0 \quad {\rm and} \\
	& \textrm{the paths } z \mapsto X_H(z) \textrm{ are a.s. continuous.}
\end{align*}
The existence and basic properties of $X_H$ are well-known \cite{Kahane}.
Consider now a diffeomorphism $F : \R^2 \to \R^2$ such that $F \in C^\infty(\R^2,\R^2)$, where $\nabla F$ has a uniform bound for the matrix norm. We define a process
\begin{equation}
	q(z,\omega) := a(z) X_H(F(z),\omega)
\end{equation}
for a deterministic function $a \in C_0^\infty(D)$ and index $H > 0$. In consequence, the covariance of $Y_H$ satisfies
\begin{equation*}
	C_{q}(z_1, z_2) = \frac 12 a(z_1) a(z_2)
 \left(|F(z_1) - F(z_0)|^{2H} + |F(z_2) - F(z_0)|^{2H} + |F(z_1) - F(z_2)|^{2H}\right).
\end{equation*}
By utilizing the Taylor expansion 
$F(z_1) = F(z_2) + (\nabla F)(z_1) (z_1 - z_2) + {\mathcal O}(|z_1-z_2|^2)$
we obtain $\sigma(C_q) \in S^{-2-2H}_{1,0}(\R^2\times \R^2)$ and
\begin{equation*}
	\sigma^p(z,\xi) \propto a(z)^2 |\nabla F(z) \xi|^{-2-2H} = b(z,\xi^0) |\xi|^{-2-2H},
\end{equation*} 
where $b(z,\xi^0) = a(z)^2 |\nabla F(z) \xi^0|^{-2-2H}$ with $\xi^0 = \xi/|\xi|$.

We recognize that $b$ is $p$-homogeneous for $p=-2-2H$ and does not fulfil assumption (A4). However, the assumption is used only in the proof of Theorem \ref{thm:recover}. In fact, we obtain the integrals $(\Rt b)(x',r)$ for any $x'\in \md'$ and $r>0$, where $\md'$ is the projection of $\md$.

%

\section{The forward problem}
\label{sec:forward}

We define the boundary to full space single layer potential $S_k:{\cal D}(\plane)\to {\cal D}'(\R^3)$ by
\begin{equation*}
	S_k\phi(x)=\int_{\plane} g_k(x-y)\phi(y) dy
\end{equation*}
for $x\in \R^3$.
We use notation $S^+_k : {\cal D}(\plane)\to {\cal D}'(\hs)$ and $S^B_k:{\cal D}(\plane)\to {\cal D}'(\plane)$ for the boundary to half-space and boundary to boundary restrictions of $S_k$, respectively. Recall that a single layer potential is continuous at the boundary $\plane$. However, the derivative has a well-known jump condition \cite{ColtonKress83}
\begin{equation}
	\label{eq:boundary_limit}
	\lim_{x_3\to 0} \frac{\p (S_k\phi)}{\p x_3}(x) = \int_{\plane} \left(\frac{\p }{\p x_3}g_k\right)(x-y) \phi(y)dy - \frac 12 \phi(x).
\end{equation}

In the following, we study uniqueness and existence in the context of a local Sobolev space. Consequently, the solutions to \eqref{eq:helmholtz}-\eqref{eq:sommerfeld_radiation}
should be understood in a weak sense.
\begin{definition}
We say that $f \in H^1_{loc}(\R^3_+)$ 
if for any $\phi \in C^\infty_0(\R^3)$ there exists $g \in H^1(\R^3)$ such that $f\phi = g$ in $\hs$.
\end{definition}

The definition above yields a Frechet space. This can be seen by considering a family of seminorms $p_j(\cdot) = \norm{\phi_j \cdot}_{H^1(\R^3)}$, where $\phi_j \in C^\infty_0(\R^3)$ and $\phi_j(x) = 1$ within a ball with radius $2^j$.

\begin{proposition}
\label{prop:homogeneous}
Let $u\in H^1_{loc}(\R^3_+)$ solve the homogeneous problem
\begin{eqnarray}
\label{eq:homogenous_problem}
(\Delta+k^2)u & = & 0,\quad\hbox{in }\R^3_+,\\ \nonumber
\frac{\p u}{\p x_3} +\df u & = & 0 \quad \hbox{on } \R^3_0,\nonumber
\end{eqnarray}
where $u$ satisfies the Sommerfeld radiation condition \eqref{eq:sommerfeld_radiation} and $\rho \in C^{0,\alpha}(D)$ for some $\alpha>0$. Then $u=0$.
\end{proposition}
\begin{proof}
Let $u=u_+$ be a solution to the problem \eqref{eq:homogenous_problem}.
We define $u_-(x',x_3) = u_+(x',-x_3)$ in $\R^3_-$, i.e. $x_3<0$, and denote
the symmetrization of $u_+$ by
\begin{equation*}
	\tilde u (x',x_3) =
	\begin{cases}
		u_+(x',x_3) & \quad x_3\geq 0, \\
		u_-(x',x_3) & \quad x_3 < 0.
	\end{cases}
\end{equation*}
By the second Green's identity we have
\begin{multline*}
	\bra \Delta \tilde u, \phi \cet_{{\mathcal D}'(\R^3) \times C^\infty_0(\R^3)}
 = \int_{\R^3_-} u_- \Delta \phi dx + \int_{\R^3_+} u_+ \Delta \phi dx \\
	= \int_{\R^3_-} \Delta u_- \phi dx + \int_{\p \R^3_-} \left( \frac{\p u_-}{\p n} \phi - u_- \frac{\p \phi}{\p n}\right) dx 
	 + \int_{\R^3_+} \Delta u_+ \phi dx - \int_{\p \R^3_+} \left( \frac{\p u_+}{\p n} \phi - u_+ \frac{\p \phi}{\p n}\right) dx.
\end{multline*}
Moreover, we have $u_+(x',0) = u_-(x',0)$ for any $x'\in \R^2$ and 
\begin{equation*}
	\frac{\p}{\p x_3} u_-(x',0)
	= \frac{\p}{\p x_3}\left(u_+(x',-x_3)\right)|_{x_3=0}
	= -\frac{\p}{\p x_3} u_+(x',0).
\end{equation*}
Now it follows that
\begin{equation}
	\label{eq:solves_helmholz}
	\bra (\Delta+k^2) \tilde u, \phi \cet_{{\mathcal D}'(\R^3) \times C^\infty_0(\R^3)}
	= -2 \int_{\p \R^3_+}  \frac{\p u_+}{\p n} \phi dx
	= -2 \int_{\plane} \df \tilde u \phi dx,
\end{equation}
and since $\df$ has a compact support, $\tilde u$ solves the Helmholtz equations outside any open neighbourhood ${\cal U}\subset \R^3$ of $\dom$. Furthermore, by standard interior regularity arguments we have that $\tilde u \in C^\infty(\R^3\setminus {\cal U})$.

For convenience, we distinguish the upper and lower half of the two-sphere by
$$\S^2_+(r) = \S^2(r)\cap \{x_3>0\} \quad  {\rm and}\quad  \S^2_-(r) = \S^2(r) \cap \{x_3<0\}.$$ 
From the radiation condition it follows that
\begin{equation}
	\label{eq:rad_cond_opened}
	0 = \lim_{r\to \infty} \int_{\S^2(r)} \left|\frac{\p\tilde u}{\p \nu} - ik\tilde u\right|^2 dS(x) 
	= \lim_{r \to \infty} \int_{\S^2(r)} \left(\left|\frac{\p\tilde u}{\p \nu}\right|^2
	+ |k|^2 |\tilde u|^2 + 2k \im \left(\tilde u \overline{\frac{\p \tilde u}{\p \nu}}\right) \right) dS(x), 
\end{equation}
where $dS(x)$ is the surface differential.
Now we have
\begin{eqnarray}
	\label{eq:im_rad_open}
	\int_{\S^2(r)} \im \left(\tilde u \overline{\frac{\p \tilde u}{\p \nu}}\right) dS(x)
	& = & \im \left(\int_{\S^2_+(r)} u_+ \overline{\frac{\p u_+}{\p \nu}} dS(x)
	- \int_{\S^2_-(r)} u_- \overline{\frac{\p u_-}{\p \nu}} dS(x) \right) \\
	& = & 2 \im \left(\int_{\S^2_+(r)} u_+ \overline{\frac{\p u_+}{\p \nu}}dS(x) \right)\nonumber
\end{eqnarray}
where $u_-(x) = u_+(-x)$. Next, integration by parts yields
\begin{equation*}
\int_{\S^2_+(r)} u_+ \overline{\frac{\p u_+}{\p \nu}}dS(x)
=  \int_{\B^2(r)} u_+ \overline{\frac{\p u_+}{\p \nu}}dS(x) \\
 - k^2 \int_{\B^3_+(r)} |u|^2 dx + \int_{\B^3_+(r)} |\nabla u|^2dx.
\end{equation*}
and due to the boundary condition we have
\begin{equation*}
	\int_{\B^2(r)} u_+ \overline{\frac{\p u_+}{\p \nu}}dS(x)
	= \int_{\B^2(r)} \rf_k |u_+|^2 dS(x)
\end{equation*}
Combining last three identities we see that integral \eqref{eq:im_rad_open} vanishes. From equation \eqref{eq:rad_cond_opened} it now follows immediately that 
%
$\lim_{r\to \infty} \int_{\S^2(r)} |\tilde u|^2dS(x) = 0$. Recall from \eqref{eq:solves_helmholz}
that $\tilde u$ is a solution to exterior Helmholtz problem 
of any open neighbourhood ${\cal U}\subset \R^3$ of $\dom$.
In consequence, the Rellich theorem yields that $\tilde u = 0$ in $\R^3\setminus {\cal U}$. By unique continuation principle \cite{ColtonKress13} we deduce that $u_+ = 0$ in $\R^3_+$.
\end{proof}

\begin{lemma}
For any $\phi\in L^2(D)$ we have
\begin{equation}
	S_k^B\phi(x) = \int_{\R^2} \exp(ix \cdot \xi) p(\xi) \hat \phi(\xi) d\xi
	\quad {\rm for} \quad x \in \plane
\end{equation}
for 
\begin{equation*}
	p(\xi) =  \frac{C(\xi)}{\sqrt{|\xi^2-k^2|}}
		\quad {\rm where} \quad
	C(\xi) = 
	\begin{cases}
		\pi, \quad & |\xi'| > k, \\
		\pi i, \quad & |\xi'| < k.
	\end{cases}
\end{equation*}
\end{lemma}

\begin{proof}
We use notation $x=(x',x_3)\in \R^3$, $x'=(x_1,x_2)$. Note that $g_k\in H^t_{loc}(\R^3)$ for $t<1$ and thus the trace $g_k(\cdot, x_3)$ is well-defined and belongs to $H^\tau_{loc}(\R^2)$ for all $\tau < \frac 12$. Now write
	\begin{equation*}
		g_k(x',0) = \lim_{\epsilon\to 0} \int_{\R^3} \frac{\exp(ix'\cdot \xi')}{\xi'^2-k^2+\xi_3^2 + i\epsilon}d\xi' d\xi_3.
	\end{equation*}
We want to calculate
\begin{equation*}
	I_\epsilon(\xi') = \int_{\R} \frac{1}{\xi'^2-k^2+\xi_3^2+i\epsilon} d\xi_3
\end{equation*}
and study the limit
\begin{equation*}
	I(\xi') = \lim_{\epsilon\to 0} I_\epsilon(\xi') = (\Ft_{x'}g_k)(\xi',0)
\end{equation*}
in the sense of generalized functions.

Assume first $a^2:=\xi'^2-k^2>0$. Now the Lebesgue's dominated convergence
yields
$I(\xi') = \int_\R (\xi_3^2+a^2)^{-1}d\xi_3$. The extension of function $f(\xi') = (\xi_3^2+a^2)^{-1}$ to the complex plane has poles at $\xi'=ia$ and $\xi'=-ia$.
Let $\gamma_R$ be positively oriented contour that goes along the real line from $-R$ to $R$ and then counterclockwise along a semicircle centered at origin from $R$ to $-R$. Since $\gamma_R$ contains the pole at $ia$, the residue theorem yields
\begin{equation*}
	\int_{\gamma_R} f(\xi) = 2\pi i\, {\rm Res}(f,ia) = \frac{\pi}{\sqrt{\xi'^2+a^2}}.
\end{equation*}
Taking $R$ to infinity gives $I(\xi')$ since, the integral over the arc decays to zero.

By similar arguments it follows for 
$k^2>\xi'^2$ and
$b:=\sqrt{k^2-\xi'^2}>0$ that
\begin{equation*}
	I_\epsilon = \int_\R \frac{d\xi_3}{\xi_3^2-(b^2-i\epsilon)}
	= \int_\R \frac{d\xi_3}{(\xi_3-b_\epsilon)(\xi_3+b_\epsilon)}
	= \frac{\pi i}{b_\epsilon},
\end{equation*}
where $b_\epsilon$ is the square-root of $b^2-i\epsilon$.
By taking $\epsilon$ to zero we obtain the claim.
\end{proof}

In the next proposition we prove a bound for the operator norm of $S^B_k$ following the strategy used in \cite{BPS14}. With that in mind, let us introduce two essential concepts. Namely, let $\{\Phi_j\}_{j=0}^\infty \subset C^\infty(\R)$ form a dyadic partition of a unity
\begin{equation*}
	\sum_{j=0}^\infty \Phi_j(t) = 1
\end{equation*}
for any $t\in \R$ such that the following conditions hold:
\begin{itemize}
	\item[(1)] ${\rm supp} (\Phi_0) \subset [-2,2]$ and
	\item[(2)] there exists $\Phi \in C^\infty(\R)$ such that ${\rm supp} (\Phi) \subset (\frac 12,2)$ and $\Phi_j(t) = \Phi(\frac{t}{2^j})$ for $j\geq 1$.
\end{itemize}
The Theorem 3.1. in \cite{BPS14} is formulated using so-called $\epsilon$-mollifiers. We reproduce the definition here for the sake of clarity.
\begin{definition}
A family of $\epsilon$-mollifiers, $\chi_\epsilon(x,y)$, defined on ${\cal A}_1\times {\cal A}_2 \subset \R^n\times \R^n$ satisfies
\begin{itemize}
	\item[i)] $\sup_{x\in {\cal A}_1} \int_{{\cal A}_2} |\chi_\epsilon(x,y)| dy \leq C$,
	\item[ii)] $|\chi_\epsilon(x,y)| \leq \frac{C_N}{\epsilon^n} 
\left(\frac{\epsilon}{|x-y|}\right)^N$ for all $N\in {\mathbb N}$ and
	\item[iii)] $|\nabla_y \chi_\epsilon(x,y)| \leq \frac{C_N}{\epsilon^{n+1}}\left(\frac{\epsilon}{|x-y|}\right)^N$ for all $N\in {\mathbb N}$.
\end{itemize}
\end{definition}
What is crucial to our treatise below is that functions $\chi_\epsilon(x,y)=\widehat{\Phi}_j(x-y)$ form a family of $\epsilon$-mollifiers for $\epsilon = 2^{-j}$. 

\begin{proposition}
\label{prop:norm_estimate}
The operator norm of $\chi_DS^B_k$ in $L^2(D)$ is bounded by
\begin{equation}
	\norm{\chi_DS^B_k}_{L^2(D)\to L^2(D)} \leq \frac{C}{\sqrt{k}}.
\end{equation}
\end{proposition}

\begin{proof}
Consider now a symbol $p(y) = |y|^2-1$ in $\R^2$. Clearly, the characteristic variety of $p$ is a unit sphere at the origin, i.e, ${\mathcal M} := p^{-1}(0)= \S^1$ with codimension $1$.
Next, we set $q(y) := C(y)|p(y)|^{-1/2}$ with $|C(y)| = C$, and prove that if $\chi_\epsilon$ is a family of $\epsilon$-mollifiers then
there exists a constant $C>0$ such that
\begin{equation}
	\label{eq:mollifier_ineq}
	f(x) := \int_{\R^2} \left|\chi_\epsilon(x,y) q(y)\right| dy \leq \frac C{\sqrt{\epsilon}}
\end{equation}
for any $x\in \R^2$.
Before we proceed let us record three useful inequalities from \cite{BPS14}:
for any $x\in\R^2$ and $r>0$ there exists a constant $C>0$ such that
\begin{equation*}
	|B(x,r)\cap {\mathcal M}| \leq Cr.
\end{equation*}
Moreover, for any $x\in \R^2$ it holds that 
\begin{equation}
	\label{eq:dist_to_M}
	|p(x)| \geq d(x,{\mathcal M}).
\end{equation}
Also, a direct consequence of \cite[Prop. 3.11.]{BPS14} is that
\begin{equation}
	\label{eq:sup_of_mollifier}
	\sup_{x\in \R^2} \int_{{\mathcal M}} |\chi_\epsilon(x-y)| d\sigma(y) \leq \frac{C}{\epsilon}.
\end{equation}
Next, we 
consider inequality \eqref{eq:mollifier_ineq} separately in and outside of a set defined by
\begin{equation}
	{\mathcal N}_\delta = \{ x\in \R^2 \; | \; d(x,{\mathcal M}) \leq \delta\}
	= \{x \in \R^2 \; | \; k-\delta \leq |x| \leq k + \delta\}.
\end{equation}

First, recall that $\sup_x\int_{\R^2} |\chi_\epsilon(x,y)|dy < C$.
Clearly, for $x$ outside ${\mathcal N}_{1/2}$, we have $q(x) \leq C$. Also,
for $x\in {\mathcal N}_{1/2} \setminus {\mathcal N}_\epsilon$ we can estimate $q(x) \leq C/\sqrt{\epsilon}$ due to \eqref{eq:dist_to_M}. Thus, inequality follows outside ${\mathcal N}_\epsilon$.

Secondly, consider \eqref{eq:mollifier_ineq} inside ${\mathcal N}_\epsilon$.
Using the same arguments as above, we see that
\begin{equation*}
\sup_{x\in {\mathcal N}_\epsilon} \int_{\R^2\setminus{\mathcal N}_\epsilon}
	|\chi_\epsilon(x,y) q(y)| dy \leq \frac C{\sqrt{\epsilon}}.
\end{equation*}
Utilizing the inequality \eqref{eq:sup_of_mollifier} we obtain
\begin{eqnarray*}
	\sup_{x\in {\mathcal N}_\epsilon} \int_{{\mathcal N}_\epsilon}
	|\chi_\epsilon(x,y) q(y)| dy 
	& = & \sup_{x\in {\mathcal N}_\epsilon} \int_{\mathcal M} \int_{1-\epsilon}^{1+\epsilon} |\chi_\epsilon(x,(r,\theta)) q(r)| rdr d\sigma(\theta) \\
	& \leq & \frac C\epsilon \int_{1-\epsilon}^{1+\epsilon} \frac{r}{\sqrt{|r^2-1|}} dr \leq \frac C{\sqrt{\epsilon}}.
\end{eqnarray*}
This proves inequality \eqref{eq:mollifier_ineq}.

We can now turn our attention to the claim. By applying \eqref{eq:mollifier_ineq} with $\epsilon = k 2^j$, one can show that 
\begin{eqnarray}
	\label{eq:part_ineq}
	\norm{{\mathcal F}(\Phi_j g_k)}_{L^\infty(\R^2)}
	& = & \sup_{\eta \in \R^2} \left| \int_{\R^2} 2^{2j} \widehat{\Phi}(2^{j}(\eta-\xi)) \frac{c(\xi)}{\sqrt{|\xi^2-k^2|}} d\xi\right| \nonumber \\
	& = & \frac 1 k\sup_{\tau \in \R^2} \left|\int_{\R^2} (k2^j)^2 \widehat{\Phi}(k2^j(\tau-\rho)) \frac{c(\rho/k)}{\sqrt{|\rho^2-1|}} d\rho \right| \nonumber \\
	& \leq & \frac{1}{k} \sqrt{k2^j} = \frac{2^{j/2}}{\sqrt{k}}.
\end{eqnarray}
Let us now write $h=\chi_\dom S_k\phi = \chi_\dom g_k \ast \phi$ for $\phi\in L^2(\dom)$ and use notation
$h_m = \Phi_m h$, $g^j_k = \Phi_j g_k$ and $\phi_\ell = \Phi_\ell \phi$ for $m,k,\ell\geq 0$. Notice that since $\phi$ is compactly supported, there are only finitely many indeces $\ell$ such that $\phi_\ell$ is non-zero.
We have now the identity
\begin{equation*}
	h_m = \Phi_m \sum_{j=0}^\infty \sum_{\ell=0}^\infty g_k^j \ast \phi_\ell.
\end{equation*}
Recall that the support of convolution is a subset of the sum of the supports. In our case, for $\Phi_m g_k^j\ast \phi_\ell = 0$ if $2^{m+1} < 2^{j-1}-2^{\ell+1}$.
In particular, this is satisfied if $j>3+\max(\ell,m)$.
Taking the Fourier transform, we can estimate
\begin{equation*}
	\norm{\hat h_m}_{L^2(\R^2)} \leq C \sum_{\ell = 0}^\infty \sum_{j=0}^{\max(\ell,m)+3}
	\norm{\hat g^j_k}_{L^\infty(\R^2)} \norm{\hat \phi_\ell}_{L^2(\R^2)}
\end{equation*}
for any $m \geq 0$.
Combining \eqref{eq:part_ineq} and the Parseval's identity we obtain
\begin{equation}
	\label{eq:BPS_ineq}
	\sup_m 2^{-m} \norm{h_m}_{L^2(\R^2)} \leq C \sum_{\ell=0}^\infty \frac{2^{\ell/2}}{\sqrt{k}} \norm{\phi_\ell}_{L^2(\R^2)} \leq \frac{C_D}{\sqrt{k}} \norm{\phi}_{L^2(\R^2)},
\end{equation}
where the last inequality follows from $D$ being bounded. Moreover, since $h$ is compactly supported, there are only finite number of non-zero $\{h_j\}$. In consequence, the left-hand side in \eqref{eq:BPS_ineq} can be bounded from below by $C\norm{h}_{L^2(\R^2)}$. This yields the claim.
\end{proof}

\begin{lemma}
Let $\chi \in C_0^\infty(\R^3)$. The operator $\chi S_k: H^{-1/2}(D) \to H^1(\R^3)$ has a bounded norm.
\end{lemma}

\begin{proof}
Let $\chi_D\in C_0^\infty(\R^3)$ be such that $\chi_D \equiv 1$ in $\B^3(R)$ where $R>0$ is large enough such that $D \subset \B^2(R)$. Now, let $\psi\in H^{-1}(\R^3)$ and $\phi \in H^{-1/2}(D)$. Let us denote the full space potential by
$G_kf = g_k * f : {\mathcal D}(\R^3) \to {\mathcal D}'(\R^3)$, where the convolution is taken in $\R^3$.
We have
\begin{eqnarray*}
	\bra \chi S_k \phi, \psi\cet & = & 
	\bra \phi, {\rm Tr}_{\R^2} \left(\chi_D G_k^*(\chi \psi)\right) \cet \\
	& \leq & C \norm{\phi}_{H^{-1/2}(\R^2)} \norm{\chi_D G_k^* (\chi \psi)}_{H^1(\R^3)} \\
	& \leq & C \norm{\phi}_{H^{-1/2}(\R^2)} \norm{\psi}_{H^{-1}(\R^3)},
\end{eqnarray*}
where we used the mapping properties of $G_k$ in \cite[Thm. 6.11.]{Mclean}. The claim now follows since $\psi$ was arbitrary.
\end{proof}

\begin{theorem}
	\label{thm:sol}
	Let us write
	\begin{equation}
		\label{eq:sol}
		u=S_k^+\phi+u_{in},
	\end{equation}
	where $\phi$ is the unique solution to the problem
	\begin{equation}
		\label{eq:potential}
		\left(\frac 12-\rf_k S^B_k\right)\phi = \rf_k u_{in}.
	\end{equation}
    in $L^2(D)$ for almost every realization of $\rf_k$.
    Then function $u \in H^1_{loc}(\R^3_+)$ is the unique solution to the problem \eqref{eq:helmholtz}-\eqref{eq:sommerfeld_radiation} almost surely.	
\end{theorem}

\begin{proof}
It is straigtforward to see that the problem (\ref{eq:helmholtz})-(\ref{eq:sommerfeld_radiation}) has at most one solution in $H^1_{loc}(\R^3_+)$ due to Proposition \ref{prop:homogeneous}. Namely, if $u_1 \neq u_2$ are two solutions, then $\tilde u = u_1 - u_2$ solves the homogeneous problem \eqref{eq:homogenous_problem}. 

Let us now consider invertibility of \eqref{eq:potential}.
Clearly, we have $\rf_k u_{in} \in L^2(D)$ almost surely. 
We factorize $\rf_k S^B_k = \rf_k (\chi_D S^B_k)$ into two components: applying $\chi_D S^B_k$, where $\chi_D$ is the characteristic function of $D\subset \R^2$, and the multiplication operator $f \mapsto \rf_k f$.
The operator $\chi_D S^B_k$ is bounded in $L^2(D)$ and has a weakly singular kernel. It is well-known that such an operator is compact in $L^2(D)$ \cite{Vainikko}. Moreover, since $\rf_k$ is almost surely H\"older continuous, the multiplication by $\rf_k$ is a bounded operation in $L^2(D)$.
In consequence, $\rf_k S^B_k$ is compact.

For the injectivity of \eqref{eq:potential} suppose $\psi \in {\mathcal N}(\frac 12 - \lambda_k S_k^B)$. Then $\tilde u = S_k^+\psi$ solves the homogeneous problem \eqref{eq:homogenous_problem} and $\tilde u = 0$.
Since $\frac{\p \tilde u}{\p x_3} = -\frac 12 \psi = 0$, we conclude that \eqref{eq:potential} is uniquely solvable.

Finally, $u$ is clearly a solution to the Helmholtz equation for $x_3>0$. 
Moreover, at the boundary we have
\begin{equation*}
	\lim_{x_3\to 0}\left(\frac{\p u}{\p x_3} + \rf_k u\right) = -\frac 12 \phi + \rf_k S^B_k \phi + \rf_k u_{in} = 0
\end{equation*}
according to \eqref{eq:boundary_limit} and \eqref{eq:potential}. This yields the result.
\end{proof}
\begin{corollary}
	\label{cor:born_series}
	Let us denote $\phi_1 = 2 \rf_k u_{in} \in L^2(D)$
	and define an iterative scheme for each $n\geq 1$ by setting
	\begin{eqnarray*}
		\phi_{n+1} & = & 2 \rf_k S^B_k (\phi_n) \quad {\rm and} \\
		u_n & = & S_k^+ \phi_n.
	\end{eqnarray*}
	There is a random index $k_0 = k_0(\omega)$ such that $k_0<\infty$ almost 	surely and, if $k\geq k_0$ then the Born series  
	\begin{equation}
		\label{eq:born_series}
		u(x; y,k) = u_{in}(x; y,k) + u_1(x;y,k) + u_2(x; y,k) + ...
	\end{equation}
	converges pointwise for any $x,y\in \md$ to the function defined in Theorem \ref{thm:sol}.
\end{corollary}

\begin{proof}
Let us write
$$\phi_n = 2^n k^{-pn} (M_\rf S^B_k)^{n-1} (\lambda u_{in}),$$
where we use notation $M_\rf : f\mapsto \lambda f$.
Since $\lambda \in C^{\alpha}(\dom)$ almost surely, the multiplication $M_\lambda : L^2(\dom) \to L^2(\dom)$ is bounded by a finite constant almost surely. Clearly, the operator $S_k^+ : L^2(\dom) \to C(\md)$ is bounded.
Due to Proposition \ref{prop:norm_estimate} it follows that 
\begin{equation}
	\label{eq:sup_un}
	\sup_{x,y\in\md}|u_n(x; y,k)| \leq C_1 C_2^{n-1} k^{-\frac 12(n-1)-np},
\end{equation}
where $C_1$ and $C_2=C_2(\omega)$ are the norm bounds of $S_k^+$ and $M_\rf S^B_k$, respectively. We point out that $C_1$ and $C_2$ are independent of $y$.
Now it follows that 
\begin{equation*}
	\sum_{n=1}^\infty \sup_{x,y\in \md} |u_n(x,y,k)| 
	\leq C_1 \frac{k^{-p}}{1-C_2 k^{-p}} \leq 2 C_1 k^{-p},
\end{equation*}
when $C_2 k^{-p} \leq \frac 12$. Consequently, 
there is a random index $k_0=k_0(\omega)$ such that $k_0<\infty$ almost surely, and for any $k\geq k_0$ the pointwise convergence in \eqref{eq:born_series} holds.
\end{proof}

%

\section{Analysis of the Born approximation}
\label{sec:born}

\subsection{Correlation between different wavelengths}

According to Corollary \ref{cor:born_series} the first order term in the Born series satisfies
\begin{equation*}
	u_1(x,y,k) = 2 S^+_k (\rf_k u_{in}(\cdot,y))(x)
	= \frac{1}{4\pi^2k^p} \int_{\plane} \frac{\exp(ik(|x-z|+|y-z|))}{|x-z||y-z|} \rf(z) dz,
\end{equation*}
since $u_{in}(z,y) = 2 g_k(z-y)$ for any $z\in \plane$.
We denote the correlation function between the Born approximation at different wavelengths by
\begin{eqnarray}
\label{eq:integral}
\I(x,y,k_1,k_2) & = & \expec\left( u_1(x,y,k_1) \overline{u_1(x,y,k_2)}\right) \nonumber\\
 & = & \frac 1{(4\pi^2)^2k_1^p k_2^p}\int_{\plane \times \plane}
 \frac{\exp(i\left(k_1\phi(z_1; x,y)-k_2\phi(z_2; x,y)\right))}{|x-z_1|\,|z_1-y|\,|x-z_2|\,|z_2-y|} 
c_\rf(z_1,z_2)\,dz_1dz_2,
\end{eqnarray}
where 
\begin{equation}\label{eq:phi}
\phi(z; x,y)=|x-z|+|z-y|.
\end{equation}
Below we introduce multiple coordinate transformations that allow the use of microlocal methods later in our analysis. Notice that we identify $\plane$ with $\R^2$. In the process, e.g., $\phi(\cdot\; ; x,y)$, $x,y\in \md$ is considered on $\R^2$, although the distance taken in \eqref{eq:phi} is in $\R^3$.

\subsubsection{Reparametrization by $\tau_{(x,y)}$}

Let us consider the phase in the integral \eqref{eq:integral}. A simple calculation shows that
\begin{equation}
	\label{eq:simpletrick}
	k_1 \phi_1 - k_2 \phi_2 
	= \frac{(\phi_1 - \phi_2)}{2}(k_1+k_2) + \frac{(\phi_1+ \phi_2 )}{2}(k_1-k_2)
\end{equation}
where $\phi_j = \phi(z_j,x,y)$, $j=1,2$.
In the following we introduce a reparametrization to \eqref{eq:integral} so that the pair $(\phi_1\pm \phi_2)/2$ play the role of two coordinates. The benefit of this change is that the dependency of \eqref{eq:integral} on the difference $k_1-k_2$ can be explicitly analysed. What is more, once we study the case $k_1=k_2$ the first part in \eqref{eq:simpletrick} controls the high frequency limit as the second part vanishes.

We denote this change of coordinates by $\tau$ and define it as a composition of two mappings. Notice carefully that $\tau$ will depend also on $x$ and $y$. 

First, denote by $\eta: \R^2 \times \R^2 \to \R^2 \times \R^2$ the mapping
\begin{equation}
	\label{eq:eta}
	\eta(v,w) = \frac 12(v+w,-v+w).
\end{equation}
We notice that $\eta^{-1}(z_1,z_2) = (z_1-z_2, z_1+z_2)$ and that for the Jacobian of $\eta$ we have ${\rm det}(J\eta) = \frac 12$.

For the second transformation consider 
the level set 
\begin{equation}
	\label{eq:phi_levelset}
	E_t = \{z\in\R^2 \;|\; \phi(z,x,y) = t\}, \quad t>0,
\end{equation}
fixed points $x,y\in\R^2$. In fact, $E_t$ describes an ellipse with focal points $x$ and $y$ and a semi-major axis $t/2$.
The idea is to parametrize $\R^2$ in terms of the ellipses $E_t$, $t\geq 0$. 
We define $\rho_{(x,y)}:\R^2\times \R^2 \to \R^2 \times \R^2$ by
\begin{equation*}
	\rho_{(x,y)}(z_1,z_2) = (\tilde \rho_{(x,y)}(z_1),\tilde \rho_{(x,y)}(z_2))
\end{equation*} 
where
\begin{equation}
\label{eq:tilderho}
\tilde \rho_{(x,y)}(z) =  \frac 1 2 \phi(z,x,y)
\begin{pmatrix}
1 \\
\hbox{arcsin}\left(e_1\cdotp f(z,x,y)\right)
\end{pmatrix} \in \R^2
\end{equation}
where $e_1 = (1,0)^\top \in \R^2$ and
$$f(z,x,y) = \frac{\nabla_{z} \phi(z; x,y)}{\|\nabla_{z} \phi(z; x,y)\|}.$$
To sum up, the first component in \eqref{eq:tilderho} corresponds to the semi-major axis of
the ellipse $E_{\phi(z,x,y)}$. The second component specifies the angle of the normal vector of the
ellipse with $e_1$ at the
point $z$. 
\begin{remark}
\label{re:rho_wellposed}
Suppose $x',y'\in \R^2$ are the projections of $x,y\in\md$. It turns out that $\tilde \rho_{(x,y)}(z)$ is constant on the segment $L_{x',y'} = \{tx'+(1-t)y'\; | \; t\in [0,1]\}$. 
Recall that we exclude the possible singularities 
in the coordinates obtained by $\rho_{(x,y)}$ by assumption in Theorem \ref{thm:data} (i).
From this point on, we assume that $L_{x',y'} \subset \plane \setminus \overline{\dom}$ holds.
\end{remark}
We are ready to define $\tau_{(x,y)} : \R^2\times \R^2 \to \R^2 \times \R^2$ by
\begin{equation*}
	\tau_{(x,y)} = \rho_{(x,y)}^{-1} \circ \eta.
\end{equation*}
Note how the first components transform: if $(v,w) =\tau_{(x,y)}^{-1}(z_1,z_2)$ then 
$$v_1 = \frac 12 (\phi(z_1,x,y)-\phi(z_2,x,y)) \quad {\rm and} \quad w_1 = \frac 12 (\phi(z_1,x,y)+\phi(z_2,x,y)).$$
In consequence, by \eqref{eq:simpletrick} we obtain
\begin{equation*}
	k_1\phi_1 -k_2\phi_2 = (k_1+k_2)v_1 +(k_1-k_2)w_1.
\end{equation*}

\subsubsection{Representation formula and asymptotics}

By definition, the correlation function $c_\rf$ is the Schwartz kernel 
of a pseudodifferential operator $C_\rf$ with a classical symbol 
$\s(x,\xi) \in S^{-2-2\epsilon}_{1,0}(\R^2\times \R^2)$.
Moreover, the support of $c_\rf$ is contained in $\dom \times \dom$.
We can write $c_\rf$ in terms of its symbol by
\begin{equation}
	\label{eq:rep1}
	c_\rf(z_1, z_2) = (2\pi)^{-2} \int_{\R^2} e^{i(z_1-z_2)\cdot \xi} \s(z_1,\xi)d\xi
\end{equation}
All symbols considered here are classical symbols $S^t_{1,0}$, $t\in\R$ \cite{Hor3}. 

Let us shortly revisit conormal distributions of H\"ormander type \cite{Hor3}.
If $X \subset \R^n$ is an open set and $S\subset X$
is a smooth submanifold of $X$, we denote by $I(X ; \; S)$ the distributions in
${\mathcal D}'(X)$ that are smooth in $X \setminus S$ and have a conormal 
singularity at $S$.
Consequently, by equation \eqref{eq:rep1} the correlation function $c_\rf$ is a conormal distribution 
in $\R^4$ of H\"ormander type having conormal singularity on the surface
$S = \{(z_1,z_2) \in \R^4\; |\; z_1 - z_2 = 0\}.$
Moreover, the set of distributions supported in a compact subset of $X$
is denoted by $I_{comp}(X; \; S)$. 
%

Below, we transform symbols on the plane in ways that depend on measurement points $x,y \in \hs$. In order to establish uniform estimates and claims also with respect to variables $x$ and $y$ we extend the covariance functions into space 
$\cspace = \plane \times \plane \times \hs \times \hs$ by 
\begin{equation}
\label{eq:extend_cov}
	c_\rf(z_1,z_2; x,y) = c_\rf(z_1,z_2) \chi(x) \chi(y),
\end{equation}
where $\chi \in C^\infty_0(\hs)$ satisfies $\chi = 1$ in $\md$ and the projection of ${\rm supp}(\chi)$ to $\plane$ is disjoint to $\dom$. Similarly, the 
surface $S$ is extended by
\begin{equation*}
	S = \{(z_1, z_2; x,y) \in \cspace \; | \; z_1-z_2 = 0\}.
\end{equation*}
Below, we use notation $\tau({\bf x}) = \tau_{(x,y)}(z_1,z_2)$ for any ${\bf x} = (z_1,z_2,x,y) \in \cspace$.
\begin{lemma}
\label{lem:representation}
There exists a unique symbol $\tilde\s = \tilde\s(w,\xi ; x,y) \in S^{-2-2\epsilon}_{1,0}(\cspace)$ such that
\begin{equation}
	\label{eq:fourierrepresentation}
	\I(x,y,k_1,k_2) = \frac 1{4 \pi^2 k_1^p k_2^p}({\mathcal F}_w \tilde\s)\left(\vect{k_2-k_1}0, -\vect{k_1+k_2}0 ; x,y\right).
\end{equation}
Moreover, $\tilde\s$ is compactly supported in the $(w,x,y)$ variables.
\end{lemma}

\begin{proof}
Let ${\bf D} \subset \cspace$ be an open set containing $\dom \times \dom \times {\rm supp}(\chi) \times {\rm supp}(\chi)$ so that 
$c_\rf \in I_{comp}({\bf D}; \; S \cap {\bf D})$. Below, we use the fact that conormal distributions
are invariant in change of coordinates. 
We have that
\begin{equation}
\label{eq:c_tau}
c_\tau:=\tau^*(c_\rf)\in I_{comp}(\tau^{-1}({\bf D});\tilde S
\cap \tau^{-1}({\bf D})),
\end{equation}
where $\tilde S = \tau^{-1}(S)=\{(v,w,x,y)\in\cspace\; |\; v=0\}$.
With this change of coordinates we get
\begin{multline} 
\label{eq:I-form}
\I(x,y,k_1,k_2) = \frac 1{(4\pi^2)^2k_1^p k_2^p}\int_{\R^4} \exp\{i((k_1+k_2)e_1\cdotp v+(k_1-k_2)e_1\cdotp w
)\}\\
 \cdot c_\tau(v,w; x,y)H_{(x,y)}(v,w)\,dvdw
\end{multline}
where $e_1=(1,0)^\top$ is the unit vector
and
\begin{equation}
	\label{eq:Hdef}
H_{(x,y)}(v,w)=\frac{\hbox{det}\,(J\tau_{(x,y)}(v,w))}{|x-z_1||z_1-y||x-z_2||z_2-y|},
\end{equation}
where $(z_1,z_2)=\tau_{(x,y)}(v,w)$
and $J\tau_{(x,y)}$ is the Jacobian of $\tau_{(x,y)}$ with respect to $(v,w)$. 
Since $H$ is smooth in $\tau^{-1}({\bf D})$ in all variables
and the class $I_{comp}(\cspace;\;\tilde S)$ is closed in multiplication
with a smooth function, we have
$c_\tau \cdot H \in I_{comp}(\cspace;\;\tilde S)$.
Using the
representation theorem of conormal distribution \cite[ Lemma 18.2.1]{Hor3},
we obtain
\begin{equation}
\label{C5-esitys B}
c_\tau(v,w; x,y) H_{(x,y)}(v,w)=\int_{\R^2}e^{iv\cdotp \xi} \tilde\s(w,\xi; x,y)\,d\xi,
\end{equation}
where
\begin{equation}
\label{eq:c5_asymp}
\tilde\s(w,\xi; x,y)
\sim \sum_{l=0}^\infty \bra -iD_v,D_\xi\cet^l
(\sigma_\tau(w,\xi; x,y)H_{(x,y)}(v,w))|_{v=0} \in S^{-2-2\epsilon}_{1,0}(\cspace).
\end{equation}
Above, $\sigma_\tau$ is the symbol of $c_\tau$.
The equality \eqref{eq:fourierrepresentation} is obtained by combining equation (\ref{C5-esitys B}) with \eqref{eq:I-form} and applying the Fourier inversion formula. The compact support of $\tilde \s$ follows from definition \eqref{eq:extend_cov} and compact support of $c_\tau$ in \eqref{C5-esitys B}.
\end{proof}

\begin{lemma}
	For $k_1, k_2 \geq 1$ the random variable
	$u_1$ satisfies uniformly for $x,y \in \md$ the estimate
	\begin{eqnarray}
		\label{eq:k1k2A}
		|\I(x,y,k_1,k_2)| & \leq & \frac{C_n}{k_1^pk_2^p(k_1+k_2)^{2+2\epsilon} (1 + |k_1-k_2|)^n} \\
		\label{eq:k1k2B}
		|\expec( u_1(x,y,k_1) u_1(x,y,k_2))| & \leq &
		C_n' (k_1+k_2)^{-n}k_1^{-p}k_2^{-p} 
	\end{eqnarray}
	where $n$ is arbitrary. 
\end{lemma}

\begin{proof}
Since $\tilde\s \in S^{-2-2\epsilon}_{1,0}(\cspace)$ is compactly supported in $(w,x,y)$-variables, we have by definition
\begin{equation}
	\label{eq:diffsymbolappr}
	|D^\alpha_w \tilde\s(w,\xi;x,y)| \leq C_\alpha (1+|\xi|)^{-2-2\epsilon}
\end{equation}
for all $|\alpha| \geq 0$, where $C_\alpha$ is independent of $(w,x,y) \in \R^2 \times \md \times \md$.
For $k_1\neq k_2$ this implies after $n$ integrations by parts
\begin{equation*}
|\I(x,y,k_1,k_2)|\leq
C_nk_1^{-p}k_2^{-p}(1+|k_1+k_2|^{2+2\epsilon})^{-1}|k_1-k_2|^{-n}
\end{equation*}
for all $n\geq 0.$ Including the case $k_1=k_2$ yields the estimate (\ref{eq:k1k2A}).

The second estimate \eqref{eq:k1k2B} follows analogously to \eqref{eq:k1k2A} since the proof of Lemma \ref{lem:representation} allows $k_2$ to be negative.
\end{proof}

The proof of the next corollary follows by identical arguments to \cite[Cor. 1]{LPS}.

\begin{corollary}
\label{cor:estimate}
Assume that $k_1,k_2>1$ and $x,y \in \md$. Then
\begin{equation*}
	\expec \left| {\rm Re}(k_1^{1+\epsilon+p} u_1(x,y,k_1)) {\rm Re}(k_2^{1+\epsilon+p} u_1(x,y,k_2))\right|
	\leq C_n (1 + |k_1-k_2|)^{-n}, \quad n>0,
\end{equation*}
where $C_n$ is independent of $x$ and $y$, and one may replace one or both of the real parts
by imaginary parts.
\end{corollary}

\subsection{Asymptotics of the correlation}

In the following we introduce a useful reparametrization of the symbol $\tilde\s$ that allows us to study
its principal symbol and, consequently, the asymptotics related to \eqref{eq:integral}.
We define $\kappa_{(x,y)} : \R^2\times \R^2 \to \R^2 \times \R^2$ by
\begin{equation*}
	\kappa_{(x,y)} = \eta^{-1} \circ \tau_{(x,y)}.
\end{equation*}
We frequently use notation $\kappa({\bf x}) = \kappa_{(x,y)}(z_1,z_2)$ for ${\bf x} = (z_1,z_2,x,y) \in \cspace$.
Let us now decompose the coordinate transform $\kappa = (\kappa_1,\kappa_2)$.
and proceed by decomposing also the
Jacobian $J\kappa$ with similar notation. The corresponding Jacobian is given by
\begin{equation}
	\label{eq:kappa_decomposition}
J\kappa=\left(\begin{array}{cc}\kappa_{11} & \kappa_{12}\\
\kappa_{21} & \kappa_{22}\\
\end{array}\right)
:=\left(\begin{array}{cc}J_v\kappa_{1} & J_v\kappa_{2}\\
J_w\kappa_{1} & J_w\kappa_{2}\\
\end{array}\right),
\end{equation}
where we have
\begin{equation*}
\kappa_{11} := J_v \kappa_1 = \left(\begin{array}{cc} \partial_{v_1} V_1 & \partial_{v_2} V_1\\
\partial_{v_1} V_2 & \partial_{v_2} V_2\\
\end{array}\right) \in \R^{2\times 2}
\quad {\rm for} \quad
\kappa_1 = (V_1,V_2)
\end{equation*}
with similar definition for terms $J_v\kappa_2$, $J_w\kappa_1$ and $J_w\kappa_2$.

\begin{lemma}
	The symbol $\tilde \s$ defined in Lemma \ref{lem:representation} has a principal symbol
	\begin{equation}
	\label{c6 principal symbol}
	\tilde \s^{p}(w,\xi; x,y) =
	\frac{b(z,(\kappa_{11}^{-\top} \xi)^0)}{(1+|(A(z)\kappa_{11}^{-\top}\xi|^2)^{1+\epsilon}}
	\left. \cdot \frac{H_{(x,y)}(v,w)}{|\det\kappa_{11}|} \right|_{v=0}
	\end{equation}
	where $z = (\tau_{(x,y)}(v,w))_1$, $\kappa_{11} = \kappa_{11}(v,w; x,y)$ and $H$ is defined by \eqref{eq:Hdef}.
\end{lemma}

\begin{proof}
Notice that the pull-back $c_\eta=\eta^*(c_\rf)$ belongs to $I(\cspace; \; \tilde S_2)$
where $\tilde S_2=\{(v,w,x,y) \; |\; v=0\}$.
Moreover, a direct substitution shows that
\begin{equation}
\label{c2-esitys}
c_\eta(v,w; x,y) = \int_{\R^2}e^{iv\cdotp \xi}\s_\eta(v,w,\xi; x,y)\,d\xi,
\end{equation}
where $\s_\eta(v,w,\xi; x,y) = \s((\eta(v,w))_1,\xi; x,y)$.
In order to find out how the symbol is transformed in the change of coordinates,
we have to represent $c_\eta$ with a symbol that does not depend on $v$.
Again, 
\cite[Lemma 18.2.1]{Hor3} yields that there exists $\tilde \s_\eta \in S^{-2}_{1,0}(\plane\times \plane)$ such that
\begin{equation}
\label{C2-esitys B}
\s_\eta(v,w; x,y)=\int_{\R^2}e^{iv\cdotp \xi} \tilde \s_\eta(w,\xi; x,y)\,d\xi ,
\end{equation}
where $\tilde \s_\eta$ has an asymptotic expansion
\begin{equation*}
\tilde \s_\eta(w,\xi; x,y)
\sim \sum_{l=0}^\infty \bra -iD_V,D_\xi\cet^l
\s_\eta(v,w,\xi; x,y)|_{v=0}\in S^{-2-2\epsilon}_{1,0}(\plane\times\plane).
\end{equation*}
Recall now that we have an identity $c_\tau=\kappa^* c_\eta$, since $\eta\circ \kappa = \tau$.
Below, we use \cite[Thm. 18.2.9.]{Hor3}
to provide a representation for $c_\tau$. Since $\kappa$ maps $ \widetilde{S}\cap \tau^{-1}({\bf D})$ onto $\tilde S_2\cap \eta^{-1}({\bf D})$, we obtain
\beq\label{c4-esitys}
c_\tau(v,w; x,y)=\int_{\R^2}e^{iv\cdotp \xi}
\s_\tau(w,\xi; x,y)\,d\xi,
\eeq
where $\s_\tau(w,\xi; x,y) \in S^{-2-2\epsilon}_{1,0}(\cspace)$.
Using the decomposition \eqref{eq:kappa_decomposition} we have that the symbol $c_\tau$ satisfies
\begin{equation} \label{princ. symb.}
\s_\tau(w,\xi; x,y)
= \frac{1}{|\hbox{det}\kappa_{11}|} \tilde \s_\eta(\kappa_2,\kappa_{11}^{-\top}\xi)
|_{v=0}
+r(w,\xi; x,y),
\end{equation}
where $\kappa_2 = \kappa_2(v,w; x,y)$, $\kappa_{11} = \kappa_{11}(v,w; x,y)$ and $r \in S^{-3}_{1,0}(\cspace)$. 

We note that the transformation rule used above in
 \cite[Thm. 18.2.9]{Hor3} is presented for half-densities.
The proof of the analogous result for distributions, however, is immediate.

Recall that the principal symbol of $C_\rf$ is given by $\s^p(z, \xi) = b(z, \xi^0) (1+|\xi|^2)^{-1-\epsilon}$ and thus
\begin{equation}
\label{c3 principal symbol}
\tilde \s_\eta^p(w,\xi)=\left.b(z, \xi^0)(1+|\xi|^2)^{-1-\epsilon}
\right|_{v=0}
\end{equation}
where $z = (\eta(v,w))_1$.
Notice how the dependence on $z$ appears due to the equation \eqref{eq:rep1}.
Plugging $\tilde \s_\eta^p$ to formula (\ref{princ. symb.}), we
see that the principal symbol of $\s_\tau(w,\xi; x,y)$ is
\begin{equation*}
\s_\tau^p(w,\xi; x,y) = \left.b(z,(\kappa_{11}^{-\top} \xi)^0)(1+|\kappa_{11}^{-\top} \xi|^2)^{-1-\epsilon}\right|_{v=0}
\cdotp J(w,x,y)
\end{equation*}
where $z = (\tau_{(x,y)}(v,w))_1$, $\kappa_{11} = \kappa_{11}(v,w,x,y)$ and
$$J(w,x,y)=|\hbox{det}\kappa_{11}(0,w,x,y)|^{-1}.$$
Finally, we obtain the result by considering the leading term in \eqref{eq:c5_asymp}.
\end{proof}

\begin{theorem}
\label{thm:asymptotics}
For $k_1 = k_2 = k$ we have the asymptotics
	\begin{equation}
		\label{eq:Iasymptotiikka}
		\I(x,y,k,k) = R(x,y) k^{-2-2\epsilon-2p} + {\mathcal O}(k^{-3-2p}),
	\end{equation}
	where $R \in C^\infty(\md \times \md)$ and we have
	\begin{equation}
	\label{eq:R}
		R(x,y)= \frac{1}{4^{2+\epsilon}\pi^2}
		\int_{\R^2} g(w,x,y)\frac{b(z(w),(\kappa_{11}(w)^\top e_1)^0)}{|x-z(w)|^2 |y-z(w)|^2}dw,
	\end{equation}
	where
	\begin{equation*}
		g(w,x,y) = \frac {\det(J \tau_{(x,y)}(v,w))|_{v=0}}{
		|\kappa_{11}(w)^{-\top} e_1|^{2+2\epsilon} |\det \kappa_{11}(w)|} .
	\end{equation*}
	Above, we denote $z(w) = (\tau_{(x,y)}(0,w))_1 = (\tau_{(x,y)}(0,w))_2$ and $\kappa_{11}(w)=\kappa_{11}(0,w,x,y)$.
\end{theorem}

\begin{proof}

To obtain the leading order asymptotics of $\I$, we consider 
the contributions of the principal symbol and the lower order 
remainder terms separately. We write
\begin{equation*}
\tilde \s(w,\xi;x,y) = \tilde \s^{p}(w,\xi;x,y)+\tilde \s_r(w,\xi;x,y),
\end{equation*}
where $\tilde \s_r(w,\xi;x,y)\in S_{1,0}^{-3}(\cspace)$ is smooth and
compactly supported  in $(w,x,y)$-variables. Thus
$|D_w^\alpha \tilde \s_r(w,\xi;x,y)|\leq C_\alpha (1+|\xi|)^{-3}$
for all multi-indices $\alpha$ and
we infer as for equation \eqref{eq:diffsymbolappr} that
\begin{equation}
\label{r_n- asymptotiikka A-formula}
|({\cal F}_w \tilde \s_r)({\bf 0},-2ke_1;x,y)|
={\cal O}((1+2|k|)^{-3}).
\end{equation}
Thus the contribution of $\tilde \s_r$ to $\I$ is estimated
by the right hand side of (\ref{r_n- asymptotiikka A-formula}).
Let us now consider the principal symbol. 
We substitute the principal symbol (\ref{c6 principal symbol}) to
formula (\ref{eq:fourierrepresentation}) and obtain
\begin{equation}
\label{vec j asympt}
\I(x,y,k,k) = 
\frac{1}{4\pi^{2} k^{2p}}{\cal F}_w \left.\left(
\frac{b(z,(e')^0)H_{(x,y)}(0,w)J(w,x,y)}{(1+4k^2|e'|^2)^{1+\epsilon}}
\right)\right|_{w={\bf 0}} 
+{\cal O}((1+2|k|)^{-3-2p}),
\end{equation}
where $e' = \kappa_{11} (0,w,x,y)^{-\top} e_1$ and $z = (\tau_{(x,y)}(0,w))_1$.
It holds for large $k$ that
\begin{equation*}
(1+4k^2|e'|^2)^{-1-\epsilon} = \left(\frac 14|e'|^{-2}k^{-2}\sum_{j=0}^\infty k^{-2j}(-4|e'|^2)^{-j}\right)^{1+\epsilon} = \frac{1}{4^{1+\epsilon}} |e'|^{-2-2\epsilon}k^{-2-2\epsilon} + {\cal O}(|k|^{-3}),
\end{equation*}
Now the result follows by applying such an estimate to equation \eqref{vec j asympt}.
\end{proof}

\begin{theorem}
The function $R$ is equation \eqref{eq:Iasymptotiikka} satisfies
	\begin{equation*}
		R(x,	x)=\frac{1}{4^{4+\epsilon}\pi^2}\int_{\dom}
		\frac {b(z,(z-x)^0)}{|z-x|^4}\,dz
		\quad {\rm for} \; x \in \md.
	\end{equation*}
\end{theorem}
\begin{proof}
The result can be obtained by simply evaluating the terms in \eqref{eq:R} for the case $x=y$.
First, let us write $z:=(\tau_{(x,x)}(0,w))_1 = (\tau_{(x,x)}(0,w))_2$.
A straightforward calculation yields
\begin{equation*}
	\kappa_{11}(0,w,x,x) =
	\begin{pmatrix}
		\cos \alpha + \alpha \sin \alpha & -\sin \alpha \\
		\sin \alpha - \alpha \cos \alpha &  \cos \alpha
	\end{pmatrix},
\end{equation*}
where $\alpha = w_2/w_1$. In particular, this implies ${\rm det}(\kappa_{11}(0,w,x,x)) = 1$ and $\kappa_{11}(0,w,x,x)^{-\top} e_1 = 1$.
Additionally, one can show that 
\begin{equation*}
\hbox{det}\,(J\tau_{(x,x)}(v,w))=\frac 14 
\quad
{\rm and} 
\quad
\left.\left({\rm det}\left(\frac {d} {dw}(\tau_{(x,x)})_1(v,w)\right)\right)^{-1}\right|_{v=0}=4.
\end{equation*}
The latter term appears, when the domain of integration in \eqref{eq:R} is transformed by $\tau_{(x,x)}$.
Finally, recall that 
\begin{equation*}
	f(z,x,x) =\frac{\nabla_{z} \phi(z,x,x)}{\|\nabla_{z} \phi(z,x,x)\|}
	= \frac{z-x}{|z-x|}.
\end{equation*}
We have
$\alpha = w_2/w_1 = \arcsin \left(e_1\cdotp f(z_1,x,x)\right)$
and, consequently,
$	\kappa_{11}(0,w,x,x)^{-\top} e_1 = 
	f(z_1,x,x)$, which yields the correct directional component in $b$.
By putting the arguments together we conclude that the claim holds.
\end{proof}

\section{Convergence of the measurement}

Let us first reproduce the important ergodic theorem needed. The following claim is obtained e.g. from \cite{Cramer}.

\begin{theorem}
	\label{thm:ergodic}
	Let $X_t$, $t\geq 0$, be a real valued stochastic process with continuous paths. Assume that for some positive constants $C,\epsilon>0$ the condition
	\begin{equation*}
		|\expec X_t X_{t+r}| \leq C (1+r)^{-\epsilon}
	\end{equation*}
	holds for all $t,r\geq 0$. Then almost surely
	\begin{equation*}
		\lim_{K\to\infty} \frac 1 K \int_1^K X_t dt = 0.
	\end{equation*}
\end{theorem}

\begin{proposition}
\label{prop:born_conv}
For any $x,y\in \md$ we have almost surely
\begin{equation*}
	\lim_{K\to \infty} \frac{1}{K-1} \int_1^K k^{2(1+\epsilon+p)} |u_1(x; y,k)|^2 dk = R(x,y).
\end{equation*}
\end{proposition}

\begin{proof}
By Theorem \ref{thm:asymptotics} we have that 
$\lim_{k\to\infty} \expec (k^{2(1+\epsilon+p)} |u_1(x; y,k)|^2) = R(x,y)$.
Let us write $Y(x,y,k) = k^{2(1+\epsilon+p)}(|u_1(x; y,k)|^2-\expec |u_1(x; y,k)|^2)$. We decompose $Y$ as
\begin{eqnarray*}
	Y(x,y,k) & = & k^{2(1+\epsilon+p)} \left( ({\rm Re} u_1(x; y,k)^2 - \expec ({\rm Re} u_1(x; y,k))^2) \right. \\
	& & +\left. ({\rm Im} u_1(x; y,k))^2 - \expec ({\rm Im} u_1(x; y,k))^2\right).
\end{eqnarray*}
One can show that
\begin{equation*}
	\expec |Y(x,y,k_1) Y(x,y,k_2)| \leq \frac{C}{1+|k_1-k_2|^2},
\end{equation*}
for any $k_1,k_2\geq 1$, by using Corollary \ref{cor:estimate} and the well-known identity \cite[Lemma 7]{LPS}
\begin{equation*}
	\expec \left((X_1^2-\expec X_1^2)(X_2^2-\expec X_2^2)\right)
	= 2 (\expec X_1 X_2 )^2,
\end{equation*}
where $X_1$ and $X_2$ are zero-mean Gaussian random variables.
The claim follows immediately from Theorem \ref{thm:ergodic}.
\end{proof}

We are ready to prove the main result regarding the forward problem.\\

\emph{Proof of Theorem \ref{thm:main_result}.} It remains to show that for any $x,y\in \md$ we have almost surely
\begin{equation}
	\label{eq:us_conv}
	\lim_{K\to \infty} \frac{1}{K-1} \int_1^K k^{2(1+\epsilon+p)} |u_s(x;y,k)|^2 dk = R(x,y).
\end{equation}
Recall from Corollary \ref{cor:born_series} that 
there is a random index $k_0=k_0(\omega)$ such that $k_0<\infty$ almost surely, and for any $k\geq k_0$ it holds that the Born series converges pointwise in $\md\times \md$. Since $u_1 \in C(\md\times \md)$, also the series $u_r = u_s - u_1$ converges pointwise. Now from inequality \eqref{eq:sup_un} we see that
\begin{equation*}
	\sum_{n\geq 2} \sup_{x,y\in \md} |u_n(x; y,k)| \leq C k^{-\frac 12 -2p},
\end{equation*}
where $C=C(\omega)$ is finite almost surely. Hence it follows that
\begin{equation}
	\label{eq:res_conv}
	\lim_{K\to\infty} \frac 1{K-1}\int_{1}^K k^{2(1+\epsilon+p)} |u_r|^2 dk \leq
	\lim_{K\to\infty} \frac C{K-1} \int_{1}^K k^{1+2\epsilon -2p} dk = 0,
\end{equation}
since $1+2\epsilon-2p<0$ according to Assumption (A2). The result is obtained by
combining Proposition \ref{prop:born_conv} together with equation \eqref{eq:res_conv} and by showing that the integral of the cross-term in
$|u_s|^2 = |u_1|^2+2\re(u_1 \overline{u_r})+|u_r|^2$ decays. Here, the Cauchy--Schwarz inequality yields the desired claim in $L^2([1,K])$ equipped with the weight $(K-1)^{-1}dk$. This proves the statement.
\hfill$\square$

\section{Recoverability}
\label{sec:recoverability}

In this section we prove that given $R(x,x)$, $x\in\md$, we can recover information about the principal symbol of $C_\rf$ and, especially, the local strength $b=b(x,\xi^0)$. We first show that the data reduces to anisotropic spherical Radon transforms 
\begin{equation}
\label{eq:rt}
	(\Rt b)(x',r) = \int_{\S^1} b(x'+r\theta,\theta) d|\theta|
\end{equation}
for $x'\in\plane$ and $r>0$.
Afterwards, we explicitly solve the null space of this transform and give proofs to our main Theorems \ref{thm:invert_S} and \ref{thm:main_result}.
On the notation: below, we often identify $\S^1$ with $[0,2\pi)$ or $\R$ modulo $[0,2\pi)$, while integration over $\S^1$ in \eqref{eq:rt}.
 
\begin{theorem}
\label{thm:recover}
The backscattering data i.e. $R(x,x)$, $x = (x',x_3) \in \md$ determines the integrals $(\Rt b)(x',r)$ for any $x' \in \plane$ and $r>0$.
\end{theorem}

\begin{proof}
Let us write
\begin{equation*}
	R(x,x)  \propto  \int_{\plane} \frac{b(y,(y-x')^0)}{|y-x|^4} dy 
	= \frac{1}{2\pi}\int_0^\infty (\Rt b)(x',r) \frac{1}{(r^2+x_3^2)^2} dr,
\end{equation*}
where $y\in\R^2$ was represented in polar coordinates $(\theta,r) \in \S^1 \times \R_+$. Recall that $R(x,x)$ is only known inside the set $\md$. However, notice that we have $(\Rt b)(x',r) = 0$ for all $r < {\rm dist}(\md',D)$, where $\md'$ is the projection of $\md$ to $\plane$ in Assumption (A1). Now, consider function $F(x_3) = R(x,x)$ for any $x = (x',x_3) \in\md$, where $x'$ is fixed. 
Let us extend $F$ to the complex plane. We find that $F$ is analytic in any disc $B(x_3,r)$ such that $r<x_3$. Consequently, we recover $F(z)$, $z\in \R$ and $z>0$, by the Taylor series
\begin{equation}
	F(z) = \sum_{n=0}^\infty \frac{F^{(n)}(x_3)}{n!} (z-x_3)^n.
\end{equation}
Moreover, function $R$ is smooth in the coordinate $x_3$ and hence it holds that
\begin{equation*}
	\left(C \frac{1}{x_3} \p_{x_3} R\right)(x,x)
	= \int_0^\infty (\Rt b)(x',r)\frac{1}{(r^2+x_3^2)^3} dr
\end{equation*}
with a suitable constant $C$. It follows that
\begin{equation*}
	\lim_{x_3\to 0} \left(C \frac{1}{x_3} \p_{x_3} R\right)(x,x)
	= \int_0^\infty (\Rt b)(x',r) r^{-4}  dr.
\end{equation*}
By applying the operator $\frac{C}{x_3} \p_{x_3}$ repeatedly, we recover
the integrals
\begin{equation*}
	\int_0^\infty \frac{(\Rt b)(x',r)}{r^2} Q\left(\frac{1}{r^2}\right) dr,
\end{equation*}
where $Q(t) = \sum_{j=0}^p a_j t^j$, $p\geq 0$. 
The support of $r\mapsto \Rt b(x',r)$ lies in a finite interval $[a,b]$ with $a,b>0$. Since functions of the form $Q(1/r^2)$ are dense in $C([a,b])$, we can uniquely determine $\Rt b(x',r)$ for all $r>0$ and any $x'\in \md'$, where $\md'$ is the projection of $\md$ to $\plane$.

The remaining task is to uniquely continue the data to $\plane \times \R_+$.
Let $\phi$ be a radial analytic function on $\R^2$. Due to assumption (A4) we notice that the function
\begin{equation*}
	h(x) = \int_{\R^2} \phi(y) b\left(x-y, \frac{y}{|y|}\right) |y|^s dy
	= \int_{\dom} \phi(x-z) \tilde b(z,z-x) dz
\end{equation*}
is holomorphic in a $\C^2$-neighbourhood of $\R^2$. Consequently, we can uniquely continue $h$ to any $x \in \R^2$. We choose 
a sequence of analytic functions $\phi_j \in C^\infty_0(0,\infty)$ such that
$\phi_j$ converges to the Dirac delta $\delta(\cdot - r_0)$ in the sense generalized functions. We can then write
\begin{equation*}
	h_j(x) = \int_{\dom} \phi_j(|x-z|) \tilde b(z,z-x) dz = \frac{1}{2\pi}\int_0^\infty h_j(r) \Rt(x,r) r^{s} dr.
\end{equation*}
By taking $j$ to zero, we obtain $\lim_{j \to \infty} 2\pi h_j(x)/ r_0^s = \Rt(x,r_0)$
for any $x\in \R^2$ and $r_0>0$. This concludes the proof.
\end{proof}
Let us now study the general properties of the following transformation $\Rt$ for compactly supported functions $f \in C^\infty_0(\R^2 \times \S^1)$.
We are interested about how $\Rt$ acts in the subspace
\begin{equation}
	\sset = \{ f \in C^\infty_0(\R^2 \times \S^1) \; | \; f(\cdot, z) = f(\cdot, -z)\}
	\subset C^\infty_0(\R^2 \times \S^1)
\end{equation}
and show the following result:
\begin{theorem}
	\label{thm:nullspace}
	The operator $\Rt : \sset \to R(\Rt)$ has a null space
\begin{equation}
	{\mathcal N}(\Rt) =  \left\{ g \in X \; | \; 
	(\Ft g)\left(\xi, T(-\theta)\frac{\xi^\bot}{|\xi|}\right)
	= -(\Ft g)\left(\xi, T(\theta)\frac{\xi^\bot}{|\xi|}\right) 
	\; \textrm{for any } 0 \leq \theta \leq \frac \pi 2\right\}
\end{equation}
where the mapping $T(\theta) : \S^1 \to \S^1$ rotates $\S^1$ rigidly by angle $\theta$ and we have denoted $\Ft = \Ft_{x\to\xi}$.
\end{theorem}

\begin{proof}
Let us assume that $f\in X$ and
\begin{equation}
	\label{eq:rt=0}
	(\Rt f) (x,r) = 0 \quad \textrm{for all } x\in \R^2, r>0.
\end{equation}
One can show that $(\Ft \Rt f)(\xi,r) = \int_{\S^1} e^{ir\theta \cdot \xi} (\Ft f)(\xi, \theta) d|\theta|$
by applying the Fubini theorem and a change of variables via $y = x+r\theta$.
Let us consider $(\Ft \Rt f)(\xi,r)$ for a fixed frequency $\xi$ and denote 
$\alpha = \arccos \left(\frac{\xi}{|\xi|}\cdot e_1\right)$
for $e_1 = (1,0)^\top$. Then by the Cosine rule we obtain
\begin{equation*}
(\Ft \Rt f)\left(\xi,\frac{r}{|\xi|}\right) =
\int_\alpha^{\alpha+2\pi} e^{ir \cos(\theta-\alpha)}(\Ft f)(\xi, \theta) rd\theta =
\int_{\S^1} e^{ir \cos \theta} g(\theta) d|\theta|,
\end{equation*}
where we have written $g(\theta) = (\Ft f)(\xi, \theta + \alpha)$. Since $f$ is in the null space of $\Rt$ we have
\begin{equation*}
	h(r,\theta) = \int_0^{2\pi} e^{ir \cos \theta} g(\theta) d\theta = 0
\end{equation*}
for any $r>0$.

The periodicity of $f$ in $X$ is inherited by $g$ as $\pi$-periodicity
$g(\theta+\pi) = g(\theta)$ for any $\theta$. Here, the value $\theta+\pi$ is considered modulo $2\pi$.
Let us now compute
\begin{equation*}
(\partial r)^j h(r,\theta)|_{r=0} = \int_0^{2\pi} (i \cos \theta)^j g(\theta) d\theta 
= \int_0^{\pi} \left((i \cos \theta)^j 
	+ (-i \cos \theta)^j \right)g(\theta) d\theta
	= 0
\end{equation*}
for any $j\in\Z_+$.
Clearly, it follows that
$\int_{-1}^1 p_j(t) \tilde g(t) dt = 0$, where $p_j(t) = t^{2j}(1+t^2)^{-1/2}$, for any $j\in \Z_+$ where $\tilde g(t) = g(\arccos t)$. Since the Taylor series of $\tilde g$ at zero can contain only the odd powered polynomials, we must have 
$\tilde g(-t) = - \tilde g(t)$ for any $0\leq t \leq 1$ and thus
\begin{equation*}
	g\left(\frac \pi 2 - \theta\right) = -g\left(\frac \pi 2 + \theta\right)
\end{equation*}
for any $0 \leq \theta \leq \frac \pi 2$. Finally, we get that for any $\xi$ we have
\begin{equation*}
	(\Ft f)\left(\xi, \frac \pi 2 + \alpha(\xi)-\theta\right) 
	= -(\Ft f)\left(\xi, \frac \pi 2 + \alpha(\xi) + \theta\right)
\end{equation*}
for $0 \leq \theta \leq \frac \pi 2$.
\end{proof}

\emph{Proof of Theorem \ref{thm:invert_S}}
In particular, Theorem \ref{thm:nullspace} implies that if $g\in {\mathcal N}(\Rt)$ then
\begin{equation}
	\label{eq:nollapisteet}
	(\Ft g)\left(\xi, (\xi^\bot)^0\right) = 
	(\Ft g)\left(\xi, \xi^0\right) = 	
	0 \quad \textrm{for all } \xi \in \R^2.
\end{equation}

Since $\Rt$ is linear, we have that $\Rt : X \setminus {\mathcal N}(\Rt) \to R(\Rt)$ is invertible in its range. Hence from the knowledge of data
\begin{equation}
	\label{eq:data}
	\{\Rt f(x,r) \; | \; x\in \R^2, r\geq 0\}
\end{equation}
we can recover the class $f + {\mathcal N}(\Rt)$, i.e. for any $\pi$-periodic $\phi \in C([-\pi/2, \pi/2])$, which is symmetric with respect to $0$ and $\pi/2$, we know
\begin{equation}
	\label{eq:duality}
	\int_{-\frac \pi 2}^{\frac \pi 2} (\Ft f)\left(\xi,T(\theta)(\xi^\bot)^0\right) \phi(\theta) d\theta
\end{equation}
for all $\xi \in \R^2 \setminus \{0\}$. In particular, the equation \eqref{eq:nollapisteet} yields that we know values
\begin{equation}
	\label{eq:knownvalues}
	(\Ft f)\left(\xi, \xi^0\right) \quad {\rm and} \quad
	(\Ft f)\left(\xi, (\xi^\bot)^0\right).
\end{equation}
for all $\xi\in \R^2$.
\hfill$\square$
\vspace{0.4cm}\\ 
\indent We now turn our attention to functions of type
\begin{equation}
	\label{eq:easier}
	f(x,y) = \langle y, A(x) y\rangle
\end{equation}
Recall the according to (A5) the matrix field $x \mapsto A(x)$ is smooth and symmetric. Moreover, it has uniformly bounded eigenvalues and satisfies ${\rm supp}(A)\subset D$. Clearly, recovering $f$ everywhere would lead to recovering $A$ everywhere. However, as we noticed above this is not possible using only the properties of $\Rt$.\vspace{0.2cm}\newline

\emph{Proof of Theorem \ref{thm:main_result}} 
Suppose the backscattering data $n_0(x)$, $x \in \md$, is given. According to Theorem \ref{thm:recover} we can uniquely recover functions $(\Rt b)(x',r)$ for all $x'\in \plane$ and $r>0$.
Let us now denote
\begin{equation*}
	A(x)
	=
	\begin{pmatrix}
		a(x) & b(x) \\
		b(x) & c(x)
	\end{pmatrix}
\end{equation*}
and $\hat A(\xi) = (\Ft A)(\xi)$.
Theorem \ref{thm:invert_S} yields the values 
\begin{equation*}
	r_1(\xi) = \frac{\langle \hat A(\xi) \xi, \xi\rangle}{|\xi|^2} 
	\quad {\rm and} \quad 
	r_2(\xi) = \frac{\langle \hat A(\xi) \xi^\bot, \xi^\bot\rangle}{|\xi|^2}
\end{equation*}
for all $\xi\in \R^2\setminus \{0\}$. In polar coordinates
$\xi = \xi(r,\theta) = r(\cos \theta, \sin \theta)^\top$ the function $r_1$ can be written as
\begin{eqnarray*}
	r_1(\xi) & = & \hat a(\xi) \cos^2\theta + 2 \hat b(\xi) \sin\theta \cos \theta
	+ \hat c(\xi) \sin^2\theta \\
	& = & \frac 12 \left(\hat a(\xi) + \hat c(\xi) + (\hat a(\xi) - \hat c(\xi)) \cos 2\theta + 2 \hat b(\xi) \sin 2 \theta\right).
\end{eqnarray*}
Likewise, $r_2$ has the form
\begin{equation*}
	r_2(\xi) = \frac 12 \left(\hat a(\xi) + \hat c(\xi) - (\hat a(\xi) - \hat c(\xi)) \cos 2\theta - 2 \hat b(\xi)\sin 2 \theta\right).
\end{equation*}
Since we know both $r_1$ and $r_2$ for $\xi \in \R^2\setminus \{0\}$, we recover
$r_1(\xi) + r_2(\xi) = {\rm Tr}(\hat A(\xi))$. Additionally, we notice that for $\xi=0$ we have $(\Ft \Rt f)(0,r) = {\rm Tr}(\hat A(0))$.
It follows that we recover $\Ft^{-1}({\rm Tr}(\hat A)) = {\rm Tr}(A)$ everywhere in $\R^2$.

Next consider equation
\begin{equation*}
	\label{eq:rec2}
	r_1(\xi) - r_2(\xi) = (\hat a(\xi) - \hat c(\xi))\cos 2\theta + 2 \hat b(\xi) \sin 2 \theta.
\end{equation*}
Suppose now that one of the functions $a$, $b$ or $c$ is known. It is easy to see that the other two are recovered outside the set $\{\cos 2\theta = 0\}$ or $\{\sin 2\theta = 0\}$. Since the components of $\hat A$ are analytic, we can recover the whole function.
\hfill$\square$

\begin{remark}
\label{re:freq_dependency}
In this work we impose a frequency-dependent model on the Robin coefficient $\rf$ in Assumption (A4). This is required in order to show convergence for the scattered residual term $u_r$ in the proof of Theorem \ref{thm:main_result}. We point out that if the residual term can be neglected, our result can also be applied directly to a frequency-independent case ($p=0$). Moreover, it remains future work to 
study if also in such a case the multiple scattering converges, i.e., 
\begin{equation}
	\lim_{K\to\infty}\frac 1{K-1}\int_1^K k^{2+2\epsilon} |u_n|^2 dk = 0
\end{equation}
for $n\geq 2$.
\end{remark}

\vspace{1cm}
{\em Acknowledgements.}
Helin and P\"aiv\"arinta were supported by the ERC-2010 Advanced Grant, 267700 - InvProb (Inverse Problems). P\"aiv\"arinta was additionally supported by the Finnish Centre of Excellence in Inverse Problems Research. Lassas was supported by the Academy of Finland. We would like to thank Petri Ola and Gunther Uhlmann for fruitful discussions.

\bibliographystyle{abbrv}

\bibliography{citations}

\end{document}